\documentclass[11pt]{amsart}
\usepackage{amsxtra,amssymb,amsmath,amsthm,amsbsy}
\usepackage[all]{xy}
\usepackage[pdftitle={Paper}, pdfauthor={BCdH}]{hyperref}
\usepackage{comment}

\usepackage{enumerate}

\usepackage[colorinlistoftodos]{todonotes}

\numberwithin{equation}{section}

\theoremstyle{plain}
\newtheorem{theorem}{Theorem}[section]
\newtheorem{lemma}[theorem]{Lemma}
\newtheorem{proposition}[theorem]{Proposition}
\newtheorem{corollary}[theorem]{Corollary}

\theoremstyle{definition}
\newtheorem{definition}[theorem]{Definition}
\newtheorem{example}[theorem]{Example}
\newtheorem{remark}[theorem]{Remark}

\newcommand{\id}         {{\mathrm {Id}}}

\def\R{\mathbb R}

\def\id{{\rm id}}

\def\then{\Rightarrow}

\def\toto{\rightrightarrows}

\def\<{\langle}
\def\>{\rangle}

\begin{document}

\title[]
{Poisson double structures}

\author[]
{Henrique Bursztyn \and Alejandro Cabrera \and Matias del Hoyo}

\address
{IMPA, Estrada Dona Castorina 110, Rio de Janeiro, 22460-320, Brazil
} \email{henrique@impa.br}

\address{Departamento de Matem\'atica Aplicada - IM, Universidade
Federal do Rio de Janeiro, CEP 21941-909, Rio de Janeiro, Brazil}
\email{alejandro@matematica.ufrj.br}

\address
{Universidade Federal Fluminense (UFF), Rua Professor Marcos Waldemar de Freitas Reis, s/n, Niterói, 24.210-201 RJ, Brazil.}
\email{mldelhoyo@id.uff.br}

\date{}


\begin{abstract}
We introduce Poisson double algebroids, and the equivalent concept of double Lie bialgebroid, which 
arise as second-order infinitesimal counterparts of Poisson double groupoids. We develop their underlying Lie theory, showing how these objects are related by differentiation and integration. We use these results to revisit Lie 2-bialgebras by means of Poisson double structures.
\end{abstract}


\maketitle

\vspace{-0.8cm}

{\begin{center}{\em To the memory of Kirill Mackenzie}\end{center}}

\setcounter{tocdepth}{1} 
\tableofcontents 

\vspace{-1cm}


\section{Introduction}

Kirill Mackenzie's pioneering work on  Lie groupoids and Lie algebroids  
was fundamental in promoting their theory to an established field of differential geometry (see \cite{Mac-book}). Among the many facets of his contributions, a central aspect was the investigation of double Lie groupoids and related ``double structures''. These objects naturally emerge e.g. in the theory of symmetries and representations of Lie groupoids and provide a lead-in to ``higher'', or ``categorified'', Lie theory. This paper further develops their  interconnections with Poisson geometry through the study of new double structures that arise as infinitesimal counterparts of Poisson double groupoids, along with differentiation and integration results relating them.

The close ties between Poisson geometry and Lie groupoids began to be unraveled in the 1980's with
the advent of {\em symplectic groupoids} \cite{Kar,We87}, originally as part of a quantization scheme, and presently a central tool in Poisson geometry. From a Lie-theoretic perspective, Poisson manifolds carry a Lie-algebroid structure, and symplectic groupoids are their global counterparts. Further connections with double structures became more apparent in the theory of {\em Poisson Lie groups} \cite{Drinfeld,LuWe2}, which also originated in quantization as \emph{semi-classical} analogues of quantum groups.
Since Poisson Lie groups are at the same time Lie groups and Poisson manifolds, they can  be simultaneously regarded as  global and infinitesimal objects. 
The differentiation of Poisson Lie groups gives rise to {\em Lie bialgebras},
which are key structures behind
the existence of {\em dual Poisson Lie groups}. 
As shown by Lu and Weinstein \cite{LuWe}, any Poisson Lie group $G$ admits an integration to a symplectic groupoid carrying a second compatible groupoid structure integrating the dual Poisson Lie group $G^*$,
$$
\begin{matrix}
(\Sigma,\omega) & \rightrightarrows & G^* \\ \downdownarrows &  & \downdownarrows\\ G & \rightrightarrows & \ast.
\end{matrix}
$$
The resulting object is an example of a {\em symplectic double groupoid}.

The introduction of {\em Poisson groupoids} by Weinstein \cite{We88} as a common framework to treat Poisson Lie groups and symplectic groupoids, indicated far-reaching generalizations of the theory of Poisson Lie groups, including their infinitesimal theory, duality and integration. Mackenzie and collaborators played a central role in advancing this program. For instance, Mackenzie and Xu established the infinitesimal-global correspondence between Lie bialgebroids and Poisson groupoids in \cite{Mac-Xu,Mac-Xu2} (illustrated in the table below). 
In \cite{mac-double,mac-notion,mac-crelle}, Mackenzie investigated Drinfeld doubles in the more general setting of Lie bialgebroids, and in \cite{mk2} the duality of Poisson groupoids was explored in the context of general symplectic double groupoids.

{\small 
\bgroup
\def\arraystretch{1.5}
\begin{figure}[h!]
\begin{center}
\begin{tabular}{ |c|c| }
\hline
Lie groupoid & Lie algebroid \\ \hline
Symplectic groupoid & (cotangent of) Poisson manifold\\ \hline
Poisson groupoid & Lie bialgebroid\\ \hline
\end{tabular}
\label{tab:macxu1}
\end{center}
\end{figure}
\egroup
}

In the course of his investigations on Poisson duality, Mackenzie introduced {\em Poisson double groupoids} \cite{mk2}, a simultaneous generalization of Poisson groupoids and symplectic double groupoids.
Other special cases include (strict) Poisson 2-groups \cite{CSX2}, and new examples arising on moduli spaces are discussed in \cite{daniel2}.

A basic ingredient in our description of the infinitesimal versions of Poisson double groupoids 
is Mackenzie's extensive work on the infinitesimal invariants of double Lie groupoids.
The differentiation of double Lie groupoids  contains two steps, one for each groupoid structure. The first leads to {\em LA-groupoids} \cite{mac-double1}, which are hybrid double structures combining Lie algebroids and Lie groupoids, while the result of the second differentiation is a  {\em double Lie algebroid} \cite{mac-double2,mac-crelle}. 
In contrast with double Lie groupoids, double Lie algebroids are not clearly ``doubles'' in a categorical sense,
which makes their general definition much subtler, see \cite{mac-double,mac-notion,mac-crelle}. 
As it turns out, the compatibilities involved in a double Lie algebroid can be effectively expressed  through the duality between Lie algebroids and linear Poisson structures, a viewpoint that will be recurrently used in this paper.

Regarding new tools, this paper relies on  recent developments in the theory of double structures, such as \cite{BCdH} and \cite{pike}, to complement Mackenzie's work on Poisson double groupoids with the study of their infinitesimal versions, along with a generalization of the Mackenzie-Xu correspondences in the previous table  to this ``categorified'' setting.

\subsection*{Main results and structure of the paper}

In the process of differentiating Poisson groupoids,
Mackenzie and Xu \cite{Mac-Xu} were led to {\em Poisson algebroids}, which are Lie algebroids carrying suitably compatible Poisson structures on their total spaces, and showed that they gave an alternative way to codify Lie bialgebroids (see $\S$ \ref{subsec:poisalg}). 

The main objects of interest in this paper are {\em Poisson double algebroids} (Definition~\ref{def:pda}), defined by a compatible Poisson structure on a {\em double} Lie algebroid. Analogously to Poisson algebroids, Poisson double algebroids may be expressed in terms of pairs of double Lie algebroids in duality, leading to the equivalent concept of {\em double Lie bialgebroid} (Definition~\ref{def:dlbalg}).

Poisson double algebroids are a natural step further from double Lie algebroids. To have a general picture, recall that a
double vector bundle $D$ has a vertical dual $D^*$ and a horizontal dual $D^\bullet$, and various double structures can be viewed in terms of Poisson structures on $D$, $D^*$ and $D^\bullet$ by means of the duality between linear Poisson structures and Lie algebroids:

\
{\footnotesize 
\bgroup
\def\arraystretch{1.5}
\begin{center}\label{tab:ds}
\begin{tabular}{ |r|c|c|c| }
 \hline
 & $D$ & $\qquad\qquad D^*\qquad \qquad$ & $D^\bullet$ \\
 \hline
  $\pi_D$  & Poisson double vector bundle & \multicolumn{2}{c|}{VB-algebroid} \\ \hline
  $\pi_D,\pi_{D^*}$ & 
  \multicolumn{2}{c|}{Poisson VB-algebroid} &
  double Lie algebroid \\ \hline
  $\pi_D,\pi_{D^*},\pi_{D^\bullet}$ & \multicolumn{3}{c|}{Poisson double algebroid} \\ \hline
\end{tabular}
\end{center}
\egroup
}

\

The paper begins
with a review of Poisson groupoids, Poisson algebroids and Lie bialgebroids, which are then used to express the compatibility conditions of various double structures
recalled in  Section \ref{sec:DS}, see the previous table. The main result in this section is Theorem~ \ref{thm:Lie2LA}, which is a version of Lie second's theorem in the context of double Lie algebroids and LA-groupoids.

Poisson double algebroids are introduced in Section \ref{sec:infinitDPG}, along with their main examples and properties, as well as the companion notion of double Lie bialgebroid. Following \cite{pike}, we present an algebraic characterization of 
these objects in terms of Weil algebras in Proposition~\ref{prop:pda-dlba}.

Our main results linking Poisson double algebroids/double Lie bialgebroids to Poisson double groupoids by differentiation and integration are in Section \ref{sec:diffint}. This ``higher'' Mackenzie-Xu correspondence is established in two steps, having {\em Poisson LA-groupoids} (see Definition~\ref{def:PLA}) as intermediate objects. Theorem~\ref{thm:diffintPLA} provides an infinitesimal-global correspondence between Poisson double algebroids and Poisson LA-groupoids, while the parallel result relating Poisson LA-groupoids and Poisson double groupoid is presented in Theorem \ref{thm:diffintegpis}. The composition of these results shows that double Poisson algebroids/Lie bialgebroids are the second-order infinitesimal invariants of Poisson double groupoids. In the symplectic setting, these results clarify the relation between symplectic double groupoids and Lie bialgebroids.

\

{\footnotesize
\bgroup
\def\arraystretch{1.5}
\begin{center}
\begin{tabular}{ |c|c|c| }
\hline
Double Lie groupoid & LA-groupoid  & Double Lie algebroid \\ \hline
Symplectic double groupoid & (cotangent of) Poisson groupoid & (cotangent of) Lie bialgebroid
\\ \hline
Poisson double groupoid & Lie-bialgebroid groupoid  & double Lie bialgebroid
\\ \hline
\end{tabular}
\end{center}
\egroup
}

\

In Section \ref{sec:lie2bi},  we consider Lie 2-bialgebras \cite{BSZ,CSX,CSX2} in light of Poisson double structures, revisiting their main properties from this perspective.


\medskip

\noindent{\bf Acknowledgments}: This project has been partially supported by CNPq and Faperj.
We thank T. Drummond for helpful advice, and  M. de Leon and I. Androulidakis for organizing this memorial volume. We are pleased to dedicate this paper to Kirill Mackenzie, whose work has given great stimulus to our research.

\section{Poisson structures on Lie groupoids and algebroids} \label{sec:prelim}

This section will briefly recall Poisson structures which are suitably compatible with Lie algebroid or groupoid structures, and how they encode various Lie-theoretic structures.

The simplest situation is that of a vector bundle $E\to M$ equipped with a {\em linear} Poisson structure $\pi \in \mathfrak{X}^2(E)$. 
The pair $(E,\pi)$ is called a {\bf Poisson vector bundle}.
There are various equivalent ways to describe the linearity condition, e.g. in terms of linear functions on $E$ being preserved by the Poisson bracket (see e.g. \cite[$\S$ 10.3]{Mac-book}). The following characterization of linear Poisson structures will be particularly useful in this paper. 
The tangent and cotangent bundles of $E$, besides being vector bundles over $E$, carry additional vector-bundle structures $TE\to TM$ and $T^*E\to E^*$; a Poisson structure $\pi$ on $E$ is {\bf linear} if $\pi^\sharp$ is linear with respect to these structures: 
\begin{equation}\label{eq:linear}
\begin{matrix}
T^*E & \stackrel{\pi^\sharp}{\longrightarrow} & TE \\ 
\downarrow &  & \downarrow\\ 
E^* & \longrightarrow & TM.
\end{matrix}
\end{equation}

A key fact that we will use recurrently is the well-known duality between Poisson vector bundles and Lie algebroids: there is a natural bijective correspondence between linear Poisson structure on $E$ and Lie-algebroid structures on $A=E^*$, see e.g. \cite[Thms.~10.3.4 and 10.3.5]{Mac-book}. In this correspondence, the anchor of $E^*$ is just the base map in \eqref{eq:linear}.  An example of a Poisson vector bundle is $E=T^*M$ with its canonical symplectic form; its dual is the tangent Lie algebroid $A=TM$. 

This duality is also expressed at the level of morphisms. 
Given vector bundles $A\to M$ and $B\to N$ and a vector-bundle map $\phi: A \to B$,  its {\bf dual relation} is a vector subbundle of $A^*\times B^*$ over the graph of $\phi|_M$ given by
$$
\mathcal{R}(\phi^*) = \{ (\phi^*(\xi),\xi)\,|\, \xi\in B^*|_{\phi(x)}, \, x\in M \} \subset A^*\times B^*.
$$
If $A$ and $B$ are Lie algebroids, then $\phi$ is a Lie-algebroid map if and only if 
its dual relation $\mathcal{R}(\phi^*)$ is a coisotropic submanifold\footnote{Recall that a submanifold $C$ in a Poisson manifold $(M,\pi)$ is called {\bf coisotropic} if $\pi^\sharp(\mathrm{Ann}(TC)) \subseteq TC$, where $\mathrm{Ann}(TC)$ is the annihilator of $TC$. Equivalently, $\mathrm{Ann}(TC) \subset T^*M$ is a Lie subalgebroid of the Lie algebroid structure on $T^*M$ induced by $\pi$.
} of $A^*\times \overline{B^*}$, see e.g. \cite[Thm.~ 10.4.9]{Mac-book} (here $\overline{B^*}$ denotes $B^*$ with the opposite Poisson structure).

\subsection{Poisson algebroids and Lie bialgebroids}\label{subsec:poisalg}

Let $A$ be a Lie algebroid over $M$; we will use the notation $A\then M$. To describe the compatibility of a Poisson structure $\pi\in \mathfrak{X}^2(A)$ with the Lie-algebroid structure, recall that $A\then M$ gives rise to  tangent and cotangent Lie algebroids, $TA\Rightarrow TM$ and $T^*A\Rightarrow A^*$, see e.g. \cite{Mac-book}.
A Poisson structure $\pi$ on the total space $A$ is {\bf infinitesimally multiplicative}  if the map $\pi^\sharp: T^*A\to TA$ is a morphism of Lie algebroids,
\begin{equation}\label{eq:piLAmap}
\begin{matrix}
T^*A & \stackrel{\pi^\sharp}{\longrightarrow} & TA \\ 
\Downarrow &  & \Downarrow\\ 
A^* & \longrightarrow & TM.
\end{matrix}
\end{equation}
The pair $(A\Rightarrow M,\pi)$ is called a {\bf Poisson algebroid}. 

Since the Poisson structure on a Poisson algebroid is linear, it is equivalent to a Lie-algebroid structure on $A^*$. A key result of Mackenzie and Xu \cite{Mac-Xu} (see also \cite[$\S$ 12.2]{Mac-book}) is that  $(A,\pi)$ is a Poisson algebroid if and ony if the pair of Lie algebroids $(A,A^*)$ forms a {\bf Lie bialgebroid}, i.e., the following compatibility holds:
\begin{equation}\label{eq:bialgebra}
\delta_A ([\xi,\eta]_{A^*}) = [\delta_A \xi, \eta]_{A^*} + [\xi, \delta_A \eta]_{A^*}, \qquad \forall\, \xi, \eta \in \Gamma(A^*),
\end{equation}
where $\delta_A$ is the Lie-algebroid differential of $A$, and we use the natural extension of the Lie bracket $[\cdot,\cdot]_{A^*}$ on $\Gamma(A^*)$ to a Gerstenhaber bracket on $\Gamma(\wedge A^*)$. 
From a slightly different perspective, it is often convenient to think of a Lie-bialgebroid structure on a vector bundle $A\to M$ as a (degree-one) differential $\delta$ and a Gerstenhaber bracket $[\cdot,\cdot]$ on $\Gamma(\wedge A^*)$ such that $\delta$ is a derivation of $[\cdot,\cdot]$ (here $\delta$ is equivalent to a Lie-algebroid structure on $A$ \cite{vaintrob}, while $[\cdot,\cdot]$ is equivalent to a Lie-algebroid structure on $A^*$).
Although not manifest in these formulations, the roles of $A$ and $A^*$ are symmetric in the definition of Lie bialgebroid \cite{kosmann} (see \cite[Sec.~12.1]{Mac-book}). A special case of interest is that of {\bf Lie bialgebras} \cite{Drinfeld}, which arise when $M$ is a point.

An important property of a Lie bialgebroid $(A,A^*)$ is that, writing  $\rho_A:A\to TM$ and $\rho_{A^*}: A^*\to TM$ for the anchors, there is an induced a Poisson structure on $M$:
\begin{equation}\label{eq:basepoisson}
\pi_M^\sharp := \rho_{A^*}\circ \rho_A^* : T^*M\to TM.
\end{equation}

\begin{example}\label{ex:symp1}
As previously mentioned, the canonical symplectic form $\omega_{can}$ on $T^*M$ makes it into a Poisson vector bundle whose dual Lie algebroid is $TM$. For a Poisson manifold $(M,\pi)$, denote by $T^*M_\pi$
the associated Lie algebroid. Its dual Poisson structure on $TM$ is the tangent lift of $\pi$, denoted by $\pi^{tan}$.
The pair $(TM,T^*M_\pi)$ is a Lie bialgebroid, and $(T^*M_\pi,\omega_{can})$ and  $(TM,\pi^{tan})$ are Poisson algebroids.
\end{example}

It turns out that every Poisson algebroid that is symplectic is the cotangent of a Poisson manifold.

\begin{proposition}\label{prop:symplalgbs}
Let $(A,A^*)$ be a Lie bialgebroid.
If $\pi_{A^*}$ is symplectic, then 
$(A^*,\pi_{A^*})$ is isomorphic to $(T^*M_\pi,\omega_{can})$, where $T^*M_\pi$ is defined by the induced Poisson structure on $M$.
\end{proposition}

This result follows from the fact that the anchor map $\rho_A: A\to TM$, besides being a Lie-algebroid map, is a Poisson morphism $(A,\pi_{A}) \to (TM,\pi_M^{tan})$, see \cite[Prop.~12.1.13]{Mac-book}. Since $\rho_A$ is the base map of 
$\pi_{A^*}^\sharp: T^* A^*\stackrel{\sim}{\to} TA^*$, it is an isomorphism of Poisson algebroids $(A,\pi_A)\cong (TM,\pi_M^{tan})$
when $\pi_{A^*}$ is symplectic. Dually, we have an isomorphism $(A^*,\pi_{A^*})\cong (T^*M_\pi,\omega_{can})$.



\subsection{Poisson groupoids and Lie theory}

Given a Lie groupoid $G\rightrightarrows M$, we denote its
Lie algebroid by $A_G$. We denote by ``Lie'' the usual Lie functor from Lie groupoids to Lie algebroids, defined on objects by $\mathrm{Lie}(G)=A_G$.

A Poisson structure $\pi$ on  $G\rightrightarrows M$ is called {\bf multiplicative} if the graph of the multiplication map  is a coisotropic submanifold of $G\times G\times \overline{G}$, where $\overline{G}$ denotes $G$ equipped with $-\pi$. The pair $(G,\pi)$ is a {\bf Poisson groupoid} \cite{We88}. Important classes of examples include Poisson-Lie groups (when $M$ is a point) and symplectic groupoids (when $\pi$ is symplectic). A multiplicative Poisson structure on a vector bundle (viewed as a groupoid with respect to fiberwise addition) is just a linear Poisson structure, so Poisson vector bundles are special cases of Poisson groupoids.

As shown by Mackenzie and Xu \cite{Mac-Xu,Mac-Xu2}, Poisson groupoids and Lie bialgebroids are related by differentiation and integration.
To see that, recall that the tangent and cotangent bundles of a Lie groupoid $G\rightrightarrows M$ carry natural Lie-groupoid structures, denoted by $TG\rightrightarrows TM$ and $T^*G\rightrightarrows A_G^*$. 
The key observation in \cite{Mac-Xu2} is then that 
a Poisson structure $\pi_G$ on $G$ is multiplicative if and only if 
\begin{equation}\label{eq:piG}
\pi_G^\sharp: T^*G\to TG
\end{equation}
is a morphism of Lie groupoids. Using the natural isomorphisms $A_{T^*G}\simeq T^*A_G$ and $A_{TG}\simeq TA_G$, one obtains a bivector field $\pi_A$ on $A_G$ by
$$
\pi_A^\sharp = \mathrm{Lie}(\pi^\sharp_G): T^*A_G\to TA_G;
$$
we will use the notation $\pi_A= \mathrm{Lie}(\pi_G)$.
It is shown in \cite{Mac-Xu} 
that 
$(A_G,\pi_A)$ is a Poisson algebroid, and hence $(A_G,A_G^*)$ is a Lie bialgebroid. 

Assuming that $G$ is source simply connected (which implies that so are $TG$ and $T^*G$),  Lie's second theorem\footnote{In this context, Lie's second theorem states that, for Lie groupoids $G_1$ and $G_2$ such that $G_1$ is source simply connected, any Lie algebroid morphism $A_{G_1} \to A_{G_2}$ is of the form $\mathrm{Lie}(\Phi)$ for a (unique) Lie-groupoid morphism $\Phi:G_1\to G_2$.} ensures that the map $\pi_G\mapsto \mathrm{Lie}(\pi_G)$ is a bijection between multiplicative Poisson structures on $G$ and infinitesimally multiplicative Poisson structures on $A_G$, see \cite{Mac-Xu2}. A special case of this result is the  correspondence between  symplectic groupoids and Poisson manifolds via differentiation and integration:
just notice that $\pi_G$ is symplectic (i.e., \eqref{eq:piG} is an isomorphism) if and only if so is $\pi_A$, in which case $(A,\pi_A)$ is the cotangent of a Poisson manifold 
by Proposition~\ref{prop:symplalgbs}.

Coisotropic subgroupoids of a Poisson groupoid $(G,\pi_G)$ are closely related to coisotropic subalgebroids of $(A_G,\pi_A)$. A result in this direction appeared in \cite[Thm.~1.5.9]{luca1} for groupoids and subgroupoids assumed to be source simply connected. We will need a refinement of this result, which follows from the next observation.

\begin{lemma}\label{lemma:1}
Suppose that $\Phi:G_1\to G_2$ is a Lie-groupoid morphism, $G_1$ is source-connected, and $H\subset G_2$ is a Lie subgroupoid. Then $\Phi(G_1)\subset H$ if and only if $\mathrm{Lie}(\Phi)(A_{G_1})\subset A_H$. 
\end{lemma}

\begin{proof}
The fact that $\Phi(G_1)\subset H$ implies that $\mathrm{Lie}(\Phi)(A_{G_1})\subset A_H$ is clear (and independent of the fact that $G_1$ is source connected).

Suppose that $\mathrm{Lie}(\Phi)(A_{G_1})\subset A_H$.
By Lie's second theorem there exists $\widetilde{\Phi}: \widetilde{G}_1\to H\subset G_2$ integrating $\mathrm{Lie}(\Phi)$, where 
$\widetilde{G}_1$ is the source-simply-connected groupoid integrating $A_{G_1}$.
Let $p:\widetilde{G}_1\to G_1$ be the natural groupoid morphism relating these Lie groupoids.
By uniqueness of the integrating morphism, we have that $i\circ \widetilde{\Phi}=\Phi \circ p:\widetilde{G}_1\to G_2$, 
where $i:H\to G_2$ is the inclusion.
It follows that $\Phi(G_1)=\Phi\circ p(\widetilde{G}_1)=i\circ \widetilde{\Phi}(\widetilde{G}_1)\subset H$.
\end{proof}

\begin{proposition}\label{prop:coisot}
Let $(G,\pi_G)$ be a Poisson groupoid and $\pi_A=\mathrm{Lie}(\pi_G)$. If a Lie subgroupoid $(S\toto N)\subset(G\toto M)$ is coisotropic then so is its Lie algebroid $A_S$ in $(A_G,\pi_A)$. Conversely, if $S$ is source-connected, then $S$ is coisotropic if so is $A_S$.
\end{proposition}


\begin{proof}
If $S\subset G$ is a Lie subgroupoid, then so is $\mathrm{Ann}(TS) \subset T^*G$, where $\mathrm{Ann}(TS)$ is the annihilator of $TS$. The canonical isomorphisms $A_{TG}\cong TA_G$ and $A_{T^*G}\cong T^*A_{G}$ identify
$A_{TS}\cong TA_S$ and
$A_{\mathrm{Ann}(TS)} \cong \mathrm{Ann}(TA_S)$, respectively. 
The result now follows from the previous lemma applied to the Lie groupoid morphism
$\pi^\sharp_G|_{\mathrm{Ann}(TS)}:\mathrm{Ann}(TS)\subset T^*G\to TG$ and the subgroupoid $TS\subset TG$ (noticing that $\mathrm{Ann}(TS)$ is source connected when $S$ is). 
\end{proof}


The previous proposition has the following corollaries.

\begin{itemize}
    \item Connected coisotropic subgroups of a Poisson Lie group are in bijection with coisotropic subalgebras of its Lie bialgebra, see e.g \cite[Prop. 2.3]{ciccoli}. (This follows since every Lie subalgebra integrates to a Lie subgroup.)
    \item Assuming $G_1$ source connected, a groupoid morphism $\Phi:(G_1,\pi_1)\to(G_2,\pi_2)$ between Poisson groupoids is Poisson if and only if it is so at the infinitesimal level \cite[Prop.~5.1.3]{BCdH} (cf. \cite[Thm.~1.5.10]{luca1}). 
    \item If $G\toto M$ is a symplectic groupoid and $S\toto N$ is a lagrangian subgroupoid, then $N\subset M$ is coisotropic; conversely, if $S\toto N$ is a source-connected subgroupoid with $N$ coisotropic, then $S\subset G$ is lagrangian, cf.  \cite[Thm.~5.4]{cattaneo}. (For this result, one uses the characterization of lagrangian subalgebroids of $(T^*M_\pi,\omega_{can})$ as conormal bundles $\mathrm{Ann}(TN)$ of coisotropic submanifolds $N \subseteq M$. )
\end{itemize}

\begin{remark}
In general, Proposition~\ref{prop:coisot} does not lead to a one-to-one correspondence between source-connected coisotropic subgroupoids and coisotropic Lie subalgebroids, due to the fact that injective Lie-algebroid morphisms may not admit an injective integration \cite{mm2}. As an example, let $G$ be the the (symplectic) pair groupoid of $({\mathbb C}^2,\omega_{can})$. Consider the map 
$
\mu:{\mathbb C}^2\to\R$, $\mu(z_1,z_2)=-\frac{1}{2}(2|z_1|^2+|z_2|^2),$ 
and take the coisotropic submanifold $N=\mu^{-1}(-1)$. Then $\mathrm{Ann}(TN)\then N$ is isomorphic (via the anchor map) to the characteristic distribution on $N$. The corresponding foliation has non-trivial holonomy (its leaf space is the ``teardrop'' orbifold),
and therefore cannot  be integrated by a subgroupoid of $G$ (which has trivial isotropies). \hfill $\diamond$
\end{remark}

\section{Double structures (from a Poisson standpoint)}\label{sec:DS}

We continue here the discussion about compatible Poisson structures, now on {\em double vector bundles} and related double structures. To begin with, we  review the notion of double vector bundle and some of their properties. 

We will follow the viewpoint from  \cite{GR} that a vector bundle $E\to M$ is completely characterized by the action $\mu_\lambda: E\to E$, $\lambda\in \mathbb{R}$, of the monoid $(\R,\cdot)$ via multiplication by scalars.
A vector-bundle map is the same as an equivariant map, and vector subbundles coincide with invariant submanifolds. From this perspective, a Poisson structure $\pi$ on $E$ is linear if and only if $(\mu_\lambda)_*\pi=\lambda\pi$ for all $\lambda\neq 0$.

A manifold $D$ carrying two vector bundle structures $D\to A$ and $D\to B$, referred to as {\em horizontal} and {\em vertical}, is a {\bf double vector bundle} if the corresponding scalar multiplications $\mu^h$ and $\mu^v$ commute: $\mu^h_\lambda \mu^v_\epsilon = \mu^v_\epsilon \mu^h_\lambda $ for all $\lambda,\epsilon$. In this case both $A=\mu^h_0(D)$ and $B=\mu^v_0(D)$ are vector bundles over $M= \mu^h_0\mu^v_0(D)$, called the {\bf side bundles}.
 We depict double vector bundles by 
 \begin{equation}\label{eq:DVB}
\begin{matrix}
D & \longrightarrow & A \\ \downarrow & & \downarrow\\ B & \longrightarrow & M.
\end{matrix}
\end{equation}

A {\bf morphism} (or map) of double vector bundles is a smooth equivariant map with respect to the horizontal and vertical multiplications by scalars, and a subbundle of a double vector bundle is an invariant submanifold with respect to both actions.

Given a double vector bundle $D$, its {\bf core} is defined by $C=\{d\in D: \mu^h_\lambda d= \mu^v_\lambda d\}$, which is also a vector bundle over $M$.
Alternatively, $C$ is the kernel of the double-vector-bundle map $D\to A\oplus B$ given by the projections on the side bundles,
$$
C \hookrightarrow D \to A\oplus B,  
$$
where $C$ is regarded as a double vector bundle with trivial sides, and $A\oplus B$ is a double vector bundle with trivial core.
This sequence splits, though non-canonically, giving an isomorphism  $D\cong A\oplus B\oplus C$. Such an isomorphism as called a {\bf splitting}.

Prototypical examples of double vector bundles include the {\bf tangent} and {\bf cotangent} double vector bundles of a vector bundle $E\to M$ \cite[Sec.~9.4]{Mac-book},
\begin{equation}\label{eq:tgcotg}
\begin{matrix}
TE & \longrightarrow & TM \\ \downarrow &  & \downarrow\\ E & \longrightarrow & M,
\end{matrix}
\qquad\;\;\;\;\;
\begin{matrix}
T^*E & \longrightarrow & E^* \\ \downarrow &  & \downarrow\\ E & \longrightarrow & M.
\end{matrix}
\end{equation}
The core of $TE$ is identified with $E$, $T^*E$ has core $T^*M$, and a splitting in these cases is the same as a linear connection on $E$.





Double vector bundles have a rich duality theory, see e.g. \cite[Sec.~9.2]{Mac-book}.
For a double vector bundle $D$ as in \eqref{eq:DVB}, we can consider its {\bf vertical dual} $D^*\to B$
and its {\bf horizontal dual}  $D ^\bullet \to A$. As it turns out, both fit into double vector bundles as follows (with core bundles indicated in the middle of the diagrams):
$$
\begin{matrix}
D^* & \longrightarrow & C^* \\ \downarrow & A^* & \downarrow\\ B & \longrightarrow & M,
\end{matrix}
\qquad\;\;\;\;\;
\begin{matrix}
D^\bullet & \longrightarrow & A \\ \downarrow & B^* & \downarrow\\ C^* & \longrightarrow & M.
\end{matrix}
$$
The tangent and cotangent double vector bundles  \eqref{eq:tgcotg}
are related by vertical duality.

By interchanging the horizontal and vertical directions one defines the {\bf flip} double vector bundle $\mathrm{fl}(D)$. 
There is a  special pairing  
$\mathrm{fl}(D^*)\times_{C^*}D^\bullet\to \R$ 
that induces an isomorphism (\cite[Thm.~9.2.2]{Mac-book})
\begin{equation}\label{eq:duality}
\mathrm{fl}(D^*)\cong (D^\bullet)^*
\end{equation}
that is $\id$ on the side bundles and $-\id$ on the core.
Therefore, modulo flips, taking successive  duals of a  double vector bundle interchanges $D$, $D^*$ and $D^\bullet$.

Just as the tangent and cotangent bundles of a vector bundle define double vector bundles, the tangent and cotangent bundles of a double vector bundle $D$ are ``triple'' vector bundles, referred to as the {\em tangent and cotangent cubes}:
\begin{equation}\label{eq:cubes}
\begin{matrix}
\xymatrix@R=6pt@C=6pt{
TD \ar[rr] \ar[dd] \ar[rd]&  & {TA} \ar'[d][dd] \ar[dr]& \\
 & D \ar[dd] \ar[rr]&  & A\ar[dd]\\
TB \ar'[r][rr] \ar[dr]&  & TM \ar[dr]& \\
& B \ar[rr]& & M,}
\end{matrix} 
\qquad \;\;\; 
\begin{matrix}
\xymatrix@R=6pt@C=6pt{
T^*D \ar[rr] \ar[dd] \ar[rd]&  & {D^\bullet} \ar'[d][dd] \ar[dr]& \\
 & D \ar[dd] \ar[rr]&  & A\ar[dd]\\
D^* \ar'[r][rr] \ar[dr]&  & C^* \ar[dr]& \\
& B \ar[rr]& & M.}
\end{matrix}  
\end{equation}


\subsection{Poisson double vector bundles and VB-algebroids}\label{subsec:VBalg}


A {\bf Poisson double vector bundle}
is a double vector bundle $D$ 
equipped with a Poisson structure $\pi$ that is {\em double linear}, namely it is linear with respect to both  horizontal and vertical vector-bundle structures.
In analogy with \eqref{eq:linear}, a double linear Poisson structure is characterized by the property that $\pi^\sharp$ is linear with respect to $T^*D\to D^*$ and $T^*D\to D^\bullet$, hence giving a morphism of cubes:
\begin{equation}\label{eq:poissoncubemap}
\begin{matrix}
\xymatrix@R=6pt@C=6pt{
T^*D \ar[rr] \ar[dd] \ar[rd]&  & {D^\bullet} \ar'[d][dd] \ar[dr]& \\
 & D \ar[dd] \ar[rr]&  & A\ar[dd]\\
D^* \ar'[r][rr] \ar[dr]&  & C^* \ar[dr]& \\
& B \ar[rr]& & M.}
\end{matrix}  
\stackrel{\pi^\sharp}{\longrightarrow}
\begin{matrix}
\xymatrix@R=6pt@C=6pt{
TD \ar[rr] \ar[dd] \ar[rd]&  & {TA} \ar'[d][dd] \ar[dr]& \\
 & D \ar[dd] \ar[rr]&  & A\ar[dd]\\
TB \ar'[r][rr] \ar[dr]&  & TM \ar[dr]& \\
& B \ar[rr]& & M.}
\end{matrix} 
\end{equation}

A (horizontal) {\bf VB-algebroid} is a double vector bundle whose horizontal vector bundle is equipped  with a Lie-algebroid structure for which the vertical multiplication by scalars is by Lie-algebroid maps. 
Vertical VB-algebroids are defined analogously. We depict horizontal and vertical VB-algebroids with diagrams as follows:
\begin{equation}\label{eq:vba}
\begin{matrix}
\Omega & \Longrightarrow & A \\ \downarrow & & \downarrow\\ B & \Longrightarrow & M,
\end{matrix}\qquad \qquad 
\begin{matrix}
\Omega & \longrightarrow & A \\ \Downarrow & & \Downarrow\\ B & \longrightarrow & M.
\end{matrix}
\end{equation}
As the notation suggests, a horizontal VB-algebroid descends to a Lie-algebroid structure on the horizontal side bundle, and similarly for a vertical VB-algebroid.  A {\bf VB-algebroid map} is a Lie-algebroid map that is linear with respect to the additional vector-bundle structure. 

\begin{remark}
VB-algebroids were introduced in 
\cite{mac-double} under the name of {\em LA-vector bundles}; the compatibility in this original formulation was that all vector-bundle structure maps be Lie-algebroid maps (see also \cite{GM}).
We observed in \cite[Thm 3.4.3]{BCdH} that it suffices to impose the compatibility with scalar multiplication.
\hfill $\diamond$
\end{remark}

Double linear Poisson structures and (horizontal) VB-algebroids are related by (horizontal) duality \cite{mac-crelle} (see also \cite{GM})
$$
\begin{matrix}
(D,\pi) & \longrightarrow & A \\ \downarrow & & \downarrow\\ B & \longrightarrow & M
\end{matrix}\;\;\; \rightleftharpoons \;\;\;  
\begin{matrix}
D^\bullet & \Longrightarrow & A \\ \downarrow & & \downarrow\\ C^* & \Longrightarrow & M.
\end{matrix}
$$
In turn, by the isomorphism \eqref{eq:duality} coming from the special pairing, the (vertical) dual of a (horizontal)  VB-algebroid is naturally a (horizontal) VB-algebroid,
$$
\begin{matrix}
\Omega & \Longrightarrow & A \\ \downarrow & & \downarrow\\ B & \Longrightarrow & M
\end{matrix}\;\;\; \rightleftharpoons \;\;\;  
\begin{matrix}
\Omega^* & \Longrightarrow & C^* \\ \downarrow & & \downarrow\\ B & \Longrightarrow & M.
\end{matrix}
$$
Of course analogous statements hold switching the horizontal and vertical directions.

\begin{example}\label{ex:tancot}
The tangent bundle $TE\then E$ of any vector bundle $E\to M$ is a VB-algebroid with respect to $TE\to TM$. If $(E,\pi)\to M$ is a Poisson vector bundle, its cotangent bundle 
$T^*E\then E$ is a VB-algebroid with respect to $T^*E\to E^*$.
\end{example}

\begin{example}
For a Lie algebroid $A\then M$, the corresponding tangent and cotangent Lie algebroids $TA\then TM$ and $T^*A\then A^*$ are VB-algebroids with respect to the vector-bundle structures $TA\to A$ and $T^*A\to A$, respectively. 
\end{example}


\begin{example}\label{ex:LArepresentations}
(a) Consider a VB-algebroid with trivial bottom vector bundle,
\begin{equation}\label{eq:rep}
\begin{matrix}
\Omega & \longrightarrow & A \\ \Downarrow & & \Downarrow\\ M & \longrightarrow & M.
\end{matrix}
\end{equation}
The core $C$ is the kernel of the top horizontal map, and we have a natural splitting $\Omega= A\oplus C$. 
The compatibility between the Lie-algebroid structure and the horizontal linear structure on $\Omega$ implies that
$[\Gamma(A),\Gamma(C)]\subseteq \Gamma(C)$, giving a representation of $A\then M$ on $C\to M$, in such a way that $\Omega$ is their semi-direct product \cite[Sec.~7.1]{Mac-book}. 

(b) A VB-algebroid with trivial core (called ``vacant''),
\begin{equation}\label{eq:rep2}
\begin{matrix}
\Omega & \longrightarrow & A \\ \Downarrow & & \Downarrow\\ B & \longrightarrow & M,
\end{matrix}
\end{equation}
is equivalent to a representation of the Lie algebroid $A\then M$ on $B\to M$, in such a way that $\Omega$ is identified with the action Lie algebroid $A\ltimes B\then B$, see \cite[Sec.~4.1]{Mac-book}.
\end{example}

So Lie-algebroid representations can be seen as particular VB-algebroids in two ways, which are related by VB-algebroid duality.
In general VB-algebroids codify more general types of representations, known as 2-term representations up to homotopy, see \cite{GM}.
The following is another special class of examples.

\begin{example}\label{ex:2term}
A VB-algebroid of the form
\begin{equation}\label{eq:2termM}
\begin{matrix}
\Omega & \Longrightarrow & A \\ \downarrow & & \downarrow\\ M & \Longrightarrow & M
\end{matrix}
\end{equation}
is completely encoded in the map $\rho|_C: C\to A$ given by the restriction of the anchor map $\rho: \Omega\to TA$ to the core.
In fact, given any vector bundles $A$ and $C$ over $M$ and a vector-bundle map $\partial: C\to A$ over the identity, 
we can regard $C\to M$ as a trivial Lie algebroid and let it act on $A$ \cite[$\S$ 4.1]{Mac-book}   by setting $\Gamma(C)\to \mathfrak{X}(A)$, $c\mapsto \partial(c)^\vee$, where $u^\vee \in \mathfrak{X}(A)$ denotes the vertical vector field on $A$ determined by $u\in \Gamma(A)$. 
The corresponding action Lie algebroid $C\oplus A \then A$ is 
of the form \eqref{eq:2termM}. Conversely any VB-algebroid as in \eqref{eq:2termM} is such an action Lie algebroid for $\partial=\rho|_C$ (c.f. \cite[Sec.~4.1]{DJO}).

\end{example}


The next result concerns {\em functorial properties} of the duality between VB-algebroids and double linear Poisson structures. For a map $\Phi$ between double vector bundles, we will write  
$\mathcal{R}(\Phi)$, $\mathcal{R}(\Phi^\bullet)$ and $\mathcal{R}(\Phi^*)$ for the graph of $\Phi$ and the horizontal and vertical dual relations, respectively.

\begin{proposition}\label{prop:doublelinear}
Let $\Omega_1$ and $\Omega_2$  be (horizontal) VB-algebroids, and let $\Phi:\Omega_1\to \Omega_2$ be a map of double vector bundles. The following are equivalent:
\begin{enumerate}[(a)]
    \item $\Phi$ is a VB-algebroid morphism;
    \item $\mathcal{R}(\Phi)\subset \Omega_1\times \Omega_2$ is a VB-subalgebroid;
    \item $\mathcal{R}(\Phi^\bullet) \subset \Omega_1^\bullet\times \overline{\Omega_2^\bullet}$ is coisotropic;
    \item $\mathcal{R}(\Phi^*)\subset \Omega_1^*\times \Omega_2^*$ is a VB-subalgebroid.
\end{enumerate}
\end{proposition}
\begin{proof}
The equivalences between $(a)$, $(b)$ and $(c)$ are well known (once we forget the vertical vector-bundle structure). We will show the equivalence between $(b)$ and $(d)$.
Denote by $\mu^v$ the vertical scalar multiplication on $\Omega_2$, and consider the double-vector-bundle map $\mu^v_{-1}\circ \Phi: \Omega_1\to \Omega_2$, noticing that this is a VB-algebroid map if and only if so is $\Phi$. Since $\mathcal{R}(\Phi^*)=\mathrm{Ann}(\mathcal{R}(\mu^v_{-1}\circ \Phi))$, the result follows from the next remark:
Given $\Omega$ a (horizontal) VB-algebroid, a double vector subbundle $\Omega'\subset \Omega$ is a VB-subalgebroid if and only if its annihilator $\mathrm{Ann}(\Omega')\subset \Omega^*$ is a VB-subalgebroid.
In fact, the double vector bundles $\Omega'$ and $\mathrm{Ann}(\Omega')$ have the same (vertical) base, given by a Lie algebroid $B'\then M'$, and the category of VB-algebroids over $B'\then M'$ is closed under kernels, cokernels and duals, see \cite[Cor.~3.4.4]{BCdH}.
\end{proof}

%

\subsection{Poisson VB-algebroids and double Lie algebroids}\label{subsec:DLA}


A (horizontal) {\bf Poisson VB-algebroid} (or a {\bf PVB-algebroid} for short)
consists of 
a VB-algebroid $\Omega$ with a Poisson structure $\pi$,
$$
\begin{matrix}
(\Omega,\pi) & \Longrightarrow & A \\ \downarrow & & \downarrow\\ B & \Longrightarrow & M,
\end{matrix}
$$
such that $\pi$ makes $\Omega\then A$ into a Poisson algebroid and $\Omega\to B$ into a Poisson vector bundle. 
Note that the horizontal dual $\Omega^\bullet$ is again a PVB-algebroid. Vertical PVB-algebroids are defined analogously.

\begin{example}
Given a Lie algebroid $A\then M$, we saw that  $T^*A\then A^*$  is a VB-algebroid with respect to $T^*A\to A$; together with the canonical symplectic structure, $T^*A$ is a PVB-algebroid. On the other hand, any Poisson structure $\pi$ such that $(A,\pi) \then M$ is a Poisson algebroid  makes $TA\then TM$ into a PVB-algebroid with respect to the tangent lift of $\pi$.
\end{example}

We will be especially concerned with the dual objects to PVB-algebroids with respect to the vector-bundle structure.
A {\bf double Lie algebroid} is a double vector bundle $\Omega$ that is both a horizontal and a vertical VB-algebroid,
\begin{equation}\label{eq:DLAdiag}
\begin{matrix}
\Omega & \Longrightarrow & A \\ \Downarrow & & \Downarrow\\ B & \Longrightarrow & M,
\end{matrix}
\end{equation}
with the compatibility condition that its (vertical) dual is a PVB-algebroid.
Morphisms of double Lie algebroids are morphisms of double vector bundles that are also morphisms of VB-algebroids in both directions.
The next proposition describes the behavior of morphisms under the (vertical) duality between double Lie algebroids and PVB-algebroids,
\begin{equation}\label{dlapvbdual}
\begin{matrix}
\Omega & \Longrightarrow & A \\ \Downarrow & & \Downarrow\\ B & \Longrightarrow & M
\end{matrix}\;\;\; \rightleftharpoons \;\;\;  
\begin{matrix}
(\Omega^*,\pi) & \Longrightarrow & C^* \\ \downarrow & & \downarrow\\ B & \Longrightarrow & M.
\end{matrix}
\end{equation}

\begin{proposition}\label{prop:DLAmap}
A double linear map $\Phi: \Omega_1\to \Omega_2$ between double Lie algebroids with cores $C_1$ and $C_2$ is a morphism of double Lie algebroids if and only if its (say, vertical) dual relation $\mathcal{R}(\Phi^*)$ is a coisotropic subalgebroid  of $\Omega_1^*\times \overline{\Omega_2^*}\then C^*_1\times C^*_2$. 
\end{proposition}

\begin{proof}
By Proposition \ref{prop:doublelinear} (parts $(c)$ and $(d)$), $\Phi$ is a morphism of horizontal VB-algebroids if and only if the relation $\mathcal{R}(\Phi^*)$ is a subalgebroid, and it is a morphism of vertical VB-algebroids if and only if this relation is coisotropic.
\end{proof}

Double Lie algebroids were introduced by Mackenzie (see e.g. \cite{mac-crelle})  as second-order infinitesimal counterparts of double Lie groupoids, see  \cite{mac-double1,mac-double2}. 
The tangent bundle of any Lie algebroid $A\then M$ is naturally a double Lie algebroid,
\begin{equation}\label{eq:dlaTA}
\begin{matrix}
TA & \Longrightarrow & A \\ \Downarrow & & \Downarrow\\ TM & \Longrightarrow & M,
\end{matrix}
\end{equation}
and the horizontal and vertical anchors in a double Lie algebroid $\Omega$, $\rho_{hor}$ and $\rho_{ver}$, are morphisms of double Lie algebroids, 
\begin{equation}\label{eq:anchordla}
\begin{matrix}
\Omega & \Longrightarrow & A \\ \Downarrow & & \Downarrow\\  B &  \Longrightarrow & M,
\end{matrix}\qquad \stackrel{\rho_{hor}}{\longrightarrow} \qquad 
\begin{matrix}
TA & \Longrightarrow & A \\ \Downarrow & & \Downarrow\\ TM & \Longrightarrow & M.
\end{matrix}
\end{equation}
Indeed, in terms of the dual PVB-algebroid $(D=\Omega^\bullet,\pi)$, $\rho_{hor}$ is the restriction  of $\pi^\sharp$ in \eqref{eq:poissoncubemap} to the right face of the cubes (in such a case of a vertical Poisson VB-algebroid, the left and right faces of the cotangent cube are double Lie algebroids).

\begin{example}\label{ex:matchedpair}
A {\it vacant} double Lie algebroid $\Omega$ is one with trivial core. With the notations of  \eqref{eq:DLAdiag}, we have $\Omega= A\oplus B$ as a double vector bundle, and by Example~\ref{ex:LArepresentations} (b), its VB-algebroid structures come from representations of $A\then M$ on $B$ and of $B\then M$ on $A$. Mackenzie showed \cite{mac-notion,mac-crelle} that such pair of representations defines a double Lie algebroid if and only if they form a {\em matched pair} \cite{mokri}.
\end{example}

\subsection{Poisson VB-groupoids and LA-groupoids}\label{subsec:LAgrp}
We will briefly recall some global versions of the infinitesimal structures discussed in $\S$~\ref{subsec:VBalg} and $\S$~\ref{subsec:DLA}. 

A {\bf VB-groupoid}
\begin{equation}\label{eq:VBgrpdiag}
\begin{matrix}
\Gamma & \rightrightarrows & A \\ 
\downarrow &  & \downarrow\\ 
H & \rightrightarrows & M
\end{matrix}
\end{equation}
consists of Lie groupoids and vector bundles as in the diagram above such the scalar multiplication of $\Gamma\to H$ is by groupoid morphisms (covering the scalar multiplication on $A\to M$). As shown in \cite{BCdH}, this is simpler but equivalent
to previous definitions \cite{mac-double1} (see also \cite{GM2}). A {\bf VB-groupoid map}  $\Gamma_1\to \Gamma_2$ is a Lie groupoid morphism that is also linear.
The connection between VB-groupoids and VB-algebroids via differentiation and integration is explained in \cite[Sec.~4]{BCdH}.

Similarly to double vector bundles, any VB-groupoid as above has a {\bf core} vector bundle $C\to M$,  defined as $\ker (s:\Gamma|_M\to A)$. The duality theory for VB-groupoids was largely developed by Mackenzie, see e.g. \cite[$\S$ 11.2]{Mac-book}.
Given a VB-groupoid $\Gamma$, its dual $\Gamma^*$ carries a VB-groupoid structure over the dual of the core:
\begin{equation}\label{eq:VBdual}
\begin{matrix}
\Gamma^* & \rightrightarrows & C^* \\ 
\downarrow &  & \downarrow\\ 
H & \rightrightarrows & M.
\end{matrix}
\end{equation}

Natural examples of VB-groupoids parallel those of VB-algebroids in Section~\ref{subsec:VBalg},
including double vector bundles (by viewing one of the vector-bundle structures as a Lie groupoid), and tangent and cotangent bundles of Lie groupoids \cite[$\S 11.2$ and $\S 11.3$]{Mac-book}. 
We single out below the global counterparts of Examples~\ref{ex:LArepresentations} and \ref{ex:2term}.

\begin{example}\label{ex:VBgrRep}
Representation of Lie groupoids  on vector bundles  can be seen as VB-groupoids in two ways: either by their semi-direct products, which are characterized by VB-groupoids with a trivial side (analogous to \eqref{eq:rep})
or by their associated action groupoids, which correspond to VB-groupoids with trivial core. Similarly to Example~\ref{ex:LArepresentations}, these types of VB-groupoids are related by duality. 
\end{example}

\begin{example}\label{ex:VBgr2term}
Just as in Example~\ref{ex:2term}, a VB-groupoid of type
$$
\begin{matrix}
\Gamma & \rightrightarrows & A \\ 
\downarrow &  & \downarrow\\ 
M & \rightrightarrows & M
\end{matrix}
$$
is equivalent to a vector-bundle map $\partial: C\to A$. 
In this case $\Gamma$ has the structure of an action groupoid for the action of the vector bundle $C\to M$ (regarded as a Lie groupoid with respect to fiberwise addition) on $A\to M$ by $c\cdot a = a + \partial (c)$.
\end{example}

In general, VB-groupoids codify representations up to homotopy on 2-term chain complexes \cite{GM2}.

For VB-groupoids $\Gamma_1$ and $\Gamma_2$ over the same base $H\toto M$, given a VB-groupoid morphism $\Phi: \Gamma_1\to \Gamma_2$ covering the identity on $H\toto M$, the dual map $\Phi^*: \Gamma_2^*\to \Gamma_1^*$ is also a VB-groupoid morphism. When the base is not the same we still have the following:

\begin{proposition}\label{prop:VBgmap}
A vector bundle map $\Phi: (\Gamma_1\to H_1) \to (\Gamma_2\to H_2)$ between VB-groupoids 
is a VB-groupoid morphism if and only if $\mathcal{R}(\Phi^*)\subset \Gamma_1^*\times \Gamma_2^*$ is a VB-subgroupoid. Moreover, $\mathcal{R}(\Phi^*)$ is source-connected provided so is $H_1$ (or $\Gamma_1)$.
\end{proposition}

\begin{proof}
A map defines a Lie groupoid morphism if and only if its graph is a subgroupoid. Then the proof is analogous to the equivalence between (b) and (d) of Prop.~\ref{prop:doublelinear}, as the kernels, cokernels and duals of VB-groupoids are also well-defined (see \cite[Sec.~3.2]{BCdH}). 
For the second assertion,
note that $\mathcal{R}(\Phi^*)$ is a VB-groupoid over the graph of $\Phi|_{H_1}: H_1\to H_2$, and in a VB-groupoid the top groupoid is source connected if and only if so is the base groupoid \cite[Rem.~3.1.1]{BCdH}.
\end{proof}

A {\bf Poisson VB-groupoid} (or {\bf PVB-groupoid} for short) is a VB-groupoid $\Gamma$ with a Poisson structure $\pi \in \mathfrak{X}^2(\Gamma)$  which is both multiplicative and linear: 
$$
\begin{matrix}
(\Gamma,\pi) & \rightrightarrows & A \\ 
\downarrow &  & \downarrow\\ 
H & \rightrightarrows & M.
\end{matrix}
$$

The dual algebroid-like object obtained by VB-groupoid duality is an {\bf LA-groupoid} \cite{mac-double1}, defined by a diagram of Lie groupoid and Lie algebroid structures, 
\begin{equation}\label{eq:LAgrp}
\begin{matrix}
\Gamma & \rightrightarrows & A \\ 
\Downarrow &  & \Downarrow\\ 
H & \rightrightarrows & M,
\end{matrix}
\end{equation}
such that the groupoid structural maps of $\Gamma\toto A$ cover those of $H\toto M$ and are Lie-algebroid morphisms 
(for the multiplication map one needs to use the fact that the space of composable arrows $\Gamma\times_A\Gamma$ inherits a Lie algebroid structure over $H\times_M H$).
As  proven by Mackenzie in \cite[Thm.~3.14]{mk2}, VB-groupoid duality gives a bijective correspondence between these structures:
\begin{equation}\label{eq:PVBLAdual}
\begin{matrix}
(\Gamma,\pi) & \rightrightarrows & A \\ \downarrow & & \downarrow\\ B & \rightrightarrows & M
\end{matrix}\;\;\; \rightleftharpoons \;\;\;  
\begin{matrix}
\Gamma^* & \rightrightarrows & C^* \\ \Downarrow & & \Downarrow\\ B & \rightrightarrows & M.
\end{matrix}
\end{equation}


Examples of PVB-groupoids include cotangent bundles of Lie groupoids (with their canonical symplectic structures) and tangent bundles of Poisson groupoids (with the tangent lift of the Poisson structure). Dually, tangent bundles of Lie groupoids and cotangent bundles of Poisson groupoids are examples of LA-groupoids. A special class of LA-groupoids known as {\em Lie 2-algebras} will be discussed in Section~\ref{sec:lie2bi}.

An {\bf LA-groupoid map}  is a morphism of Lie groupoids $\Gamma_1\to \Gamma_2$ which is also a morphism of Lie algebroids. Combining Propositions~\ref{prop:doublelinear} and \ref{prop:VBgmap} we obtain

\begin{proposition}\label{prop:LAgmap}
A vector bundle map $\Phi: (\Gamma_1\to H_1) \to (\Gamma_2\to H_2)$ between LA-groupoids 
is an LA-groupoid morphism if and only if $\mathcal{R}(\Phi^*)\subset \Gamma_1^*\times \overline{\Gamma_2^*}$ is a coisotropic VB-subgroupoid.
\end{proposition}

Mackenzie has shown that the differentiation of the groupoid structures on an LA-groupoid results in a double Lie algebroid \cite{mac-double2}:
\begin{equation}\label{eq:LAgrLie}
\begin{matrix}
\Gamma & \rightrightarrows & A \\ 
\Downarrow &  & \Downarrow\\ 
H & \rightrightarrows & M
\end{matrix}\qquad \stackrel{\mathrm{Lie}}{\mapsto} \qquad 
\begin{matrix}
\mathrm{Lie}(\Gamma) & \Longrightarrow & A \\ 
\Downarrow &  & \Downarrow\\ 
\mathrm{Lie}(H) & \Longrightarrow & M.
\end{matrix}
\end{equation}
This result can be also derived from the dual fact  \cite[Sec.~5.3]{BCdH} that  
the differentiation of a Poisson VB-groupoid gives a Poisson VB-algebroid. Using that these objects in the dual picture admit a simple description in terms of ``multiplication by scalars'', we prove  the corresponding integration result \cite[Thm.~5.3.5]{BCdH}: If the top algebroid $\Omega\then A$ of a double Lie algebroid $\Omega$ is integrable, its source-simply-connected integration $\Gamma\toto A$ is naturally an LA-groupoid.
For later use, we now extend this differentiation/integration theorem to include morphisms, which can be seen as a version of Lie second's theorem for LA-groupoids and double Lie algebroids.

\begin{theorem}\label{thm:Lie2LA}
(1) The differentiation of an LA-groupoid morphism $\Gamma_1\to \Gamma_2$ is a morphism of double Lie algebroids $\mathrm{Lie}(\Gamma_1) \to \mathrm{Lie}(\Gamma_2)$. (2) If $\Gamma_1$ is a source-simply connected LA-groupoid, then
a double-Lie-algebroid morphisms $\mathrm{Lie}(\Gamma_1) \to \mathrm{Lie}(\Gamma_2)$ integrates to a unique LA-groupoid morphisms  $\Gamma_1\to \Gamma_2$.
\end{theorem}

\begin{proof}
Consider the LA-groupoids $(\Gamma_i\toto A_i)$ over $(H_i\toto M_i)$, for $i=1,2$, as in \eqref{eq:LAgrp}, and denote the corresponding double Lie algebroids obtained by differentiation by
$(\Omega_i\then A_i)$ over $(B_i\then M_i)$, $i=1,2$.

Consider a morphism of LA-groupoids $\Phi:\Gamma_1\to \Gamma_2$.  Since it is a morphism of VB-groupoids,  $\phi:=\mathrm{Lie}(\Phi): \Omega_1\to \Omega_2$ is a morphism of VB-algebroids
\cite[Prop.~4.3.6]{BCdH}, in particular double linear.
By Prop.~\ref{prop:LAgmap}, $\Phi$ is a morphism of LA-groupoids if and only if $\mathcal{R}(\Phi^*)\subset \Gamma_1^*\times \overline{\Gamma_2^*}$ is a coisotropic subgroupoid. 
By Prop.~\ref{prop:coisot}, $\mathrm{Lie}(\mathcal{R}(\Phi^*)) = \mathcal{R}(\phi^*)$ is a coisotropic subalgebroid of $\Omega_1^*\times \overline{\Omega}_2^*$ (see \cite[Prop.~5.3.4]{BCdH} for the commutation of the Lie and dualization functors), which implies that $\phi$ is a morphism of double Lie algebroids by Prop.~\ref{prop:DLAmap}. This proves (1).

We now prove part (2).
If $\Gamma_1$ is source simply connected and $\phi: \Omega_1\to \Omega_2$ is any morphism of double Lie algebroids, by the usual 
Lie second's theorem for Lie algebroids there is a unique Lie-groupoid morphism $\Phi: (\Gamma_1\toto A_1) \to (\Gamma_2\toto A_2)$
integrating $\phi: (\Omega_1\then A_1) \to (\Omega_2\then A_2)$. The fact that $\Phi$ is actually a morphism of LA-groupoids is verified by reversing the arguments in part (1): $\Phi$ is a vector bundle map $(\Gamma_1\to H_1) \to (\Gamma_2\to H_2)$ by 
\cite[Prop.~4.3.6]{BCdH}, and since $\mathrm{Lie}(\mathcal{R}(\Phi^*))\subset \Omega_1^*\times \overline{\Omega}_2^*$ is a coisotropic subalgebroid, 
$\mathcal{R}(\Phi^*)\subset \Gamma_1^*\times \Gamma_2^*$, which is source connected by Prop.~\ref{prop:VBgmap}, is a coisotropic subgroupoid by Prop.~\ref{prop:coisot}.
\end{proof}

\begin{remark}\label{rem:lie2more}
The integration direction in part (2) of the theorem is based on the following key fact, shown in the proof: if $\Gamma_1,\Gamma_2$ are LA-groupoids, $\Gamma_1$ source connected, and $\Phi: (\Gamma_1\toto A_1) \to (\Gamma_2\toto A_2)$ is a Lie groupoid morphism whose differentiation is a morphism of double Lie algebroids, then $\Phi$ is an LA-groupoid morphism. \hfill $\diamond$
\end{remark}

\section{Poisson double algebroids}\label{sec:infinitDPG}

Taking a step further from PVB-algebroids and double Lie algebroids,
we now introduce our main object of interest: Poisson double algebroids.


\subsection{Definition and examples}

\begin{definition}\label{def:pda}
A {\bf Poisson double algebroid} is a double Lie algebroid $\Omega$ 
equipped with a Poisson structure $\pi \in \mathfrak{X}^2(\Omega)$,
\begin{equation}\label{eq:poisDLA}
\begin{matrix}
(\Omega,\pi) & \Longrightarrow & A \\ \Downarrow & & \Downarrow\\ B & \Longrightarrow & M,
\end{matrix}
\end{equation}
making $\Omega\then A$ and $\Omega\then B$ into Poisson (VB-)algebroids.
\end{definition}

Note that both duals $\Omega^*$ and $\Omega^\bullet$ inherit Poisson double  algebroid structures, so all structures induced by duality  are of the same type. {\bf Morphisms} of Poisson double algebroids are morphisms of the underlying double Lie algebroids that preserve the Poisson structures.

\begin{example}
A Poisson algebroid $(A,\pi)$ can be regarded as a Poisson double algebroid by considering the zero-rank (say, horizontal) Lie algebroid $\mathrm{id}: A\to A$.
We can also regard PVB-algebroids, or Poisson double vector bundles, as examples of Poisson double Lie algebroids (with one, or both, Lie algebroid structures on $\Omega$ being trivial). Double Lie algebroids are particular cases as well, where the Poisson structure is zero.
\end{example}

\begin{example}
For a Poisson algebroid $(A,\pi)$, its tangent double Lie algebroid \eqref{eq:dlaTA} is a double Poisson algebroid with respect to the tangent lift of $\pi$,
\begin{equation}\label{eq:PDATA}
\begin{matrix}
(TA,\pi^{tan}) & \Longrightarrow & A \\ \Downarrow & & \Downarrow\\ TM & \Longrightarrow & M.
\end{matrix}
\end{equation}
Since the tangent bundle of a Poisson manifold $M$ is a Poisson algebroid (see Example~\ref{ex:symp1}), the double tangent bundle $TTM$ is a Poisson double algebroid. Since $T^*M$ is also a Poisson algebroid, $TT^*M$ is a Poisson double algebroid. 
\end{example}

\begin{example}\label{ex:bialgebroid}
A Poisson structure $\pi$ on a Lie algebroid $A\then M$ defines a Poisson algebroid (equivalently, a Lie bialgebroid $(A,A^*)$) if and only if $T^*A \then A^*$ fits into a double Lie algebroid \cite[Thm.~5.1]{mac-crelle}:
\begin{equation}\label{eq:symplPDLA}
\begin{matrix}
T^*A & \Longrightarrow & A \\ \Downarrow & & \Downarrow\\ A^* & \Longrightarrow & M,
\end{matrix}
\end{equation}
where the vertical Lie-algebroid structure comes from the identification $T^*A\cong T^*A^*$.
This {\em cotangent double of a Lie bialgebroid}, as called by Mackenzie, is compatible with the canonical symplectic structure on $T^*A$, hence it is a Poisson (in fact, symplectic) double algebroid.
Its horizontal dual is $TA$ as in the previous example, and its vertical dual is $TA^*$, with the tangent lift of the Poisson structure on $A^*$.

\end{example}

\begin{example}
 Lie 2-bialgebras as in \cite{BSZ,CSX} can be seen as particular Poisson double algebroids, see $\S$~\ref{subsec:lie2bi}.
\end{example}

We show in the next section that Poisson double algebroids are the infinitesimal counterparts of 
Poisson double groupoids \cite[Def.~2.2]{mk2}.

\subsection{Basic properties}

The definition of Poisson double algebroid can be directly formulated in terms of the map $\pi^\sharp: T^*\Omega\to T\Omega$. For a double Lie algebroid $\Omega$ as in \eqref{eq:DLAdiag},
its cotangent and tangent bundles carry double-Lie algebroid structures
\begin{equation}\label{eq:DLATT*}
\begin{matrix}
T^*\Omega & \Longrightarrow & \Omega^\bullet \\ 
\Downarrow &  & \Downarrow\\ 
\Omega^* & \Longrightarrow & C^*,
\end{matrix}
\qquad\quad
\begin{matrix}
T\Omega & \Longrightarrow & TA \\ 
\Downarrow &  & \Downarrow\\ 
TB  & \Longrightarrow & TM,
\end{matrix}
\end{equation}
and a Poisson structure $\pi$ makes it into a Poisson double algebroid if and only if $\pi^\sharp: T^*\Omega\to T\Omega$ is a morphism of  these double Lie algebroids. 
Moreover, 
the cotangent cube of $(\Omega,\pi)$ (see \eqref{eq:cubes}) is naturally a ``triple Lie algebroid'' (as is the tangent cube of any double Lie algebroid), and $\pi^\sharp$ is a morphism of such structures:
\begin{equation}\label{eq:poissonTLAmap}
\begin{matrix}
\xymatrix@R=6pt@C=6pt{
T^*\Omega \ar@{=>}[rr] \ar@{=>}[dd] \ar@{=>}[rd]&  & {\Omega^\bullet} \ar@{=>}'[d][dd] \ar@{=>}[dr]& \\
 & \Omega \ar@{=>}[dd] \ar@{=>}[rr]&  & A\ar@{=>}[dd]\\
\Omega^* \ar@{=>}'[r][rr] \ar@{=>}[dr]&  & C^* \ar@{=>}[dr]& \\
& B \ar@{=>}[rr]& & M.}
\end{matrix}  
\stackrel{\pi^\sharp}{\longrightarrow}
\begin{matrix}
\xymatrix@R=6pt@C=6pt{
T\Omega \ar@{=>}[rr] \ar@{=>}[dd] \ar@{=>}[rd]&  & {TA} \ar@{=>}'[d][dd] \ar@{=>}[dr]& \\
 & \Omega \ar@{=>}[dd] \ar@{=>}[rr]&  & A\ar@{=>}[dd]\\
TB \ar@{=>}'[r][rr] \ar@{=>}[dr]&  & TM \ar@{=>}[dr]& \\
& B \ar@{=>}[rr]& & M.}
\end{matrix} 
\end{equation}

Let $(\Omega,\pi)$ be a Poisson double algebroid as in (\ref{eq:poisDLA}). 
Then $A,B$ and $C^*$ are Lie algebroids over $M$, and since $(\Omega,\Omega^*)$, $(\Omega,\Omega^\bullet)$ and $(\Omega^*,\Omega^\bullet)$ are Lie bialgebroids, we have induced Poisson structures $\pi_A,\pi_B$ and $\pi_{C^*}$ on $A,B$ and $C^*$ (see \eqref{eq:basepoisson}).

\begin{proposition}\label{prop:properties}
The sides $(A,\pi_A)$ and $(B,\pi_B)$, and the core dual $(C^*,\pi_{C^*})$, are Poisson algebroids, and hence $(A,A^*)$, $(B,B^*)$ and $(C,C^*)$ are Lie bialgebroids. 
\end{proposition}

\begin{proof}
We will verify the claim for the side $B$. Since the anchor $\rho_{ver}: \Omega\to TB$ is a map of double Lie algebroids, see \eqref{eq:anchordla}, its vertical dual is a morphism of horizontal VB-algebroids:
$$
\begin{matrix}
T^*B & \Longrightarrow & B^* \\ \downarrow & & \downarrow\\  B &  \Longrightarrow & M,
\end{matrix}\qquad \stackrel{\rho^*_{ver}}{\longrightarrow} \qquad 
\begin{matrix}
\Omega^* & \Longrightarrow & C^* \\ \downarrow & & \downarrow\\ B & \Longrightarrow & M.
\end{matrix}
$$
Similarly the vertical anchor $\rho_{*,ver}$ of $\Omega^*$ is also a map of double Lie algebroids, and hence in particular a morphism of horizontal VB-algebroids:
$$
\begin{matrix}
\Omega^* & \Longrightarrow & C^* \\ \downarrow & & \downarrow\\  B &  \Longrightarrow & M,
\end{matrix}\qquad \stackrel{\rho_{*,ver}}{\longrightarrow} \qquad 
\begin{matrix}
TB & \Longrightarrow & TM \\ \downarrow & & \downarrow\\ B & \Longrightarrow & M.
\end{matrix}
$$
Hence $\pi_B^\sharp = \rho_{*,ver} \circ \rho^*_{ver}$ is a morphism from $T^*B\then B^*$ to $TB\then TM$.
\end{proof}

Motivated by the special role of symplectic double groupoids in the theory of Poisson double groupoids, we now discuss Poisson double algebroids whose Poisson structures are {\em symplectic}. 
It turns out that any symplectic double algebroid must be like \eqref{eq:symplPDLA}, in analogy with the characterization of symplectic algebroids in Prop.~\ref{prop:symplalgbs}. 

\begin{proposition}\label{prop:PDAsymp}
A Poisson double algebroid $\Omega$ for which the Poisson structure $\pi$ is symplectic is isomorphic to the cotangent of its side bialgebroid $(A,A^*)$:
$$
\begin{matrix}
(T^*A,\omega_{can}) & \Longrightarrow & A \\ \Downarrow & & \Downarrow\\ A^* & \Longrightarrow & M.
\end{matrix}
$$
\end{proposition}

Of course, by symmetry, the side $A$ is interchangeable with $B$.

\begin{proof} 
Consider the map $\pi^\sharp: T^*\Omega\to T\Omega$ as in \eqref{eq:poissonTLAmap}. Its restriction to the right faces of the cotangent and tangent cubes is a map of
 double Lie algebroids,
\begin{equation}\label{eq:sympPDAanchor}
\begin{matrix}
\Omega^\bullet & \Longrightarrow & A \\ \Downarrow & & \Downarrow\\ C^* & \Longrightarrow & M
\end{matrix}\;\; \longrightarrow \;\;
\begin{matrix}
TA & \Longrightarrow & A \\ \Downarrow & & \Downarrow\\ TM & \Longrightarrow & M.
\end{matrix}
\end{equation}
This map also preserves the Poisson structures on $\Omega^\bullet$ (dual to $\Omega\then A$) and $TA$ (tangent lift of $\pi_A$), i.e., it is a morphism of double Poisson algebroids; this follows since \eqref{eq:sympPDAanchor} is nothing but the anchor  of $\Omega^\bullet\then A$ and  $(\Omega\then A, \Omega^\bullet\then A)$ is a Lie bialgebroid, see the comments after Prop.~\ref{prop:symplalgbs}.
When $\pi$ is symplectic, \eqref{eq:sympPDAanchor} is an isomorphism. The theorem now follows by dualizing this map with respect to horizontal duals. 
\end{proof}

We thus obtain a natural bijective correspondence between Poisson algebroids (or Lie bialgebroids) and symplectic double Lie algebroids.

\subsection{Double Lie bialgebroids and the Weil-algebra viewpoint}\label{subsec:weil}

Just as Poisson algebroids are equivalent to Lie bialgebroids, one can alternatively think of Poisson double algebroids in terms of  pairs of double Lie algebroids in duality. 


\begin{definition}\label{def:dlbalg}
Let $D$ be a double vector bundle. We say that $(D,D^*)$ is a  {\bf double Lie bialgebroid} if $D$ and $D^*$ are equipped with double-Lie-algebroid structures,
$$
\begin{matrix}
D & \Longrightarrow & A \\ \Downarrow & & \Downarrow\\ B & \Longrightarrow & M,
\end{matrix}
\qquad \begin{matrix}
D^* & \Longrightarrow & C^* \\ \Downarrow & & \Downarrow\\ B & \Longrightarrow & M,
\end{matrix}
$$
such that the horizontal VB-algebroids are in duality, while the vertical Lie algebroids are compatible in the sense that $(D\then B,D^*\then B)$ is a Lie bialgebroid. 
\end{definition}

It is a direct consequence of the duality between double Lie algebroids and Poisson VB-algebroids \eqref{dlapvbdual} that a Poisson-double-algebroid structure on $D$ is equivalent to a double Lie bialgebroid $(D,D^*)$.
We next consider these objects from an algebraic perspective that extends the description of  Lie bialgebroids in terms of differentials and Gerstenhaber brackets, recalled in $\S$~\ref{subsec:poisalg}. For that we will need to consider the   {\em Weil algebra} \cite{pike,voronov} of a double vector bundle, as a generalization of the exterior algebras of vector bundles.

The {\bf Weil algebra} of a double vector bundle $D$ \eqref{eq:DVB} is a super-commutative version of the algebra of double polynomial functions (i.e., polynomial with respect to both horizontal and vertical vector bundles). 
If $D=A\oplus B\oplus C$ is a split double vector bundle then its Weil algebra is $\mathcal{W}(D):=\Gamma(\wedge A^*\otimes\wedge B^*\otimes\vee C^*)$, where $\wedge$ and $\vee$ denote the exterior and symmetric algebra functors. Since any double vector bundle $D$ admits a splitting, its Weil algebra is (non canonically) isomorphic to the previous one. An intrinsic description that does not depend on any splitting is
$$
\mathcal{W}(D)=\Gamma(\wedge A^*\otimes\wedge B^*\otimes\vee \widehat{C^*})/\sim
$$
where $\widehat{C^*}$ is the vector bundle whose space of sections is $C^\infty_{[1,1]}(D)$, the space of double linear functions on $D$, and 
$ \alpha\otimes\beta\otimes 1 \sim 1\otimes 1\otimes \alpha\beta$.
The Weil algebra $\mathcal{W}(D)$ is a bi-graded super-commutative algebra, where generating sections $\alpha$, $\beta$ and $\gamma$ of $A^*$, $B^*$ and $\widehat{C^*}$ have bi-degrees $(1,0)$, $(0,1)$ and $(1,1)$, respectively.

As shown in \cite[Theorem~6.1]{pike}, a vertical VB-algebroid structure on a double vector bundle $D$ is equivalently described by any of the following structures:
 \begin{itemize}
     \item a differential  of bidegree $(1,0)$ on  $\mathcal{W}(D)$;
     \item a differential of bidegree $(1,0)$ on $\mathcal{W}(D^\bullet)$;
     \item a Gerstenhaber bracket of bidegree $(-1,-1)$ on  $\mathcal{W}(D^*)$.
 \end{itemize}
There is an analogous characterization for horizontal VB-algebroids. 
In turn, by \cite[Theorem~8.2]{pike}, 
a double-Lie-algebroid structure on $D$ is equivalent to any of the following structures:
\begin{itemize}
    \item two commuting differentials on $\mathcal{W}(D)$ of bidegrees $(1,0)$ and $(0,1)$;
    \item a differential of bidegree $(0,1)$ and a Gerstenhaber bracket of bidegree $(-1,-1)$ on $\mathcal{W}(D^*)$ such that the differential is a derivation of the bracket;
    \item  a  differential of bidegree $(1,0)$ and a Gerstenhaber bracket of bidegree $(-1,-1)$ on $\mathcal{W}(D^\bullet)$ such that the differential is a derivation of the bracket.
\end{itemize}

We can use these results to reformulate  the compatibility in the definition of double Lie bialgebroids in an algebraic fashion. A pair $D$ and $D^*$ of double Lie algebroids defines a double Lie bialgebroid $(D,D^*)$ if and only if (i) their horizontal VB-algebroids are in duality, and (ii) the differential $\delta$ on $\mathcal{W}(D)$ induced by $D$ is a derivation of the bracket $[\cdot,\cdot]$ induced by $D^*$ (here $\delta$ and $[\cdot,\cdot]$ correspond to vertical structures). 
We then have the following:

\begin{proposition}\label{prop:pda-dlba}
The following structures are equivalent:
\begin{itemize}
    \item a Poisson double algebroid $(D,\pi)$,
    \item a double Lie bialgebroid $(D,D^*)$,
    \item  two commuting differentials $\delta_h$ and $\delta_v$  on $\mathcal{W}(D)$, of respective bidegrees $(1,0)$ and $(0,1)$, together with a Gerstenhaber bracket $[\cdot,\cdot]$ of bidegree $(-1,-1)$, so that each differential is a derivation of $[\cdot,\cdot]$.
   \end{itemize}
\end{proposition}



Note that,  by viewing a vector bundle $A\to M$ as a double vector bundle (with $B=C=0$), it is clear from Def.~\ref{def:dlbalg} that a double-Lie-bialgebroid structure on $A$ is just that of a usual Lie bialgebroid. From the Weil-algebra perspective, $\mathcal{W}(A)=\Gamma(\wedge A^*)$. A double-Lie-bialgebroid structure on $A$ must have a trivial differential, so it boils down to a degree 1 differential and a Gerstenhaber bracket on  $\Gamma(\wedge A^*)$ which are compatible.




\section{Lie theory of Poisson double structures}\label{sec:diffint}

In this section, we introduce Poisson LA-groupoids, review Poisson double groupoids \cite{mk2}, and develop the Lie theory relating them to Poisson double algebroids.

\subsection{Poisson LA-groupoids}\label{subsec:PLA}

\begin{definition}\label{def:PLA}
A {\bf Poisson LA-groupoid} is given by an LA-groupoid $\Gamma$ together with a Poisson structure $\pi\in \mathfrak{X}^2(\Gamma)$,
\begin{equation}\label{eq:poissonLAgrp}
\begin{matrix}
(\Gamma,\pi) & \rightrightarrows & A \\ 
\Downarrow &  & \Downarrow\\ 
H & \rightrightarrows & M,
\end{matrix}
\end{equation}
such that $(\Gamma,\pi)\then H$ is a Poisson algebroid and $(\Gamma,\pi)\toto A$ is a Poisson groupoid.
\end{definition}

It follows that $H$ and $A$ inherit Poisson structures turning $H\rightrightarrows M$ into a Poisson groupoid and $A\then M$ into a Poisson algebroid, respectivelly. The dual $\Gamma^*$ is naturally a Poisson LA-groupoid.

\begin{example}
LA-groupoids and Poisson VB-groupoids are particular cases of Poisson LA-groupoids, one with trivial Poisson structure, the other with trivial Lie-algebroid structure.
\end{example}

\begin{example}
Given a Poisson groupoid $(G,\pi)$, then $\Gamma=TG$ (equipped with the tangent lift of $\pi$) and $\Gamma^*=T^*G$ (with the canonical symplectic form) are examples of Poisson LA-groupoids in duality.
\end{example}

One can re-express the definition of Poisson LA-groupoid as follows. If $\Gamma$ is an LA-groupoid as in \eqref{eq:LAgrp}, its cotangent and tangent bundles have  LA-groupoid structures
\begin{equation}\label{eq:LAGTT*}
\begin{matrix}
T^*\Gamma & \rightrightarrows & \mathrm{Lie}(\Gamma)^* \\ 
\Downarrow &  & \Downarrow\\ 
\Gamma^* & \rightrightarrows & C^*,
\end{matrix}\qquad\quad 
\begin{matrix}
T\Gamma & \rightrightarrows & TA \\ 
\Downarrow &  & \Downarrow\\ 
TH & \rightrightarrows & TM,
\end{matrix}
\end{equation}
where $C$ is the core of $\Gamma$.
A Poisson structure $\pi$ makes $(\Gamma,\pi)$ into a Poisson LA-groupoid if and only if $\pi^\sharp: T^*\Gamma\to T\Gamma$ is an LA-groupoid morphism.
Moreover, since all spaces in the LA-groupoids $T^*\Gamma$ and $T\Gamma$ are Lie algebroids over the corresponding spaces of $\Gamma$, the map $\pi^\sharp$ is a morphism of {\em double Lie algebroid groupoids}: 
\begin{equation}\label{eq:poissonLALAmap}
\begin{matrix}
\xymatrix@R=6pt@C=6pt{
T^*\Gamma \ar@<0.5ex>[rr] \ar@<-0.5ex>[rr]   \ar@{=>}[dd] \ar@{=>}[rd]&  & {\mathrm{Lie}(\Gamma)^*} \ar@{=>}'[d][dd] \ar@{=>}[dr]& \\
 & \Gamma \ar@{=>}[dd] \ar@<0.5ex>[rr] \ar@<-0.5ex>[rr] &  & A\ar@{=>}[dd]\\
\Gamma^* \ar@<0.5ex>|!{[ru];[rd]}{\hole}[rr]
\ar@<-0.5ex>|!{[ru];[rd]}{\hole}[rr]
\ar@{=>}[dr]&  & C^* \ar@{=>}[dr]& \\
& H \ar@<0.5ex>[rr] \ar@<-0.5ex>[rr] & & M.}
\end{matrix}  
\stackrel{\pi^\sharp}{\longrightarrow}
\begin{matrix}
\xymatrix@R=6pt@C=6pt{
T\Gamma \ar@<0.5ex>[rr] \ar@<-0.5ex>[rr] \ar@{=>}[dd] \ar@{=>}[rd]&  & {TA} \ar@{=>}'[d][dd] \ar@{=>}[dr]& \\
 & \Gamma \ar@{=>}[dd] \ar@<0.5ex>[rr] \ar@<-0.5ex>[rr] &  & A\ar@{=>}[dd]\\
TH \ar@<0.5ex>|!{[ru];[rd]}{\hole}[rr]
\ar@<-0.5ex>|!{[ru];[rd]}{\hole}[rr]
\ar@{=>}[dr]&  & TM \ar@{=>}[dr]& \\
& H \ar@<0.5ex>[rr] \ar@<-0.5ex>[rr] & & M.}
\end{matrix} 
\end{equation}

When the Poisson structure $\pi$ is symplectic,
$\pi^\sharp$ restricts to an isomorphism between the bottom faces in \eqref{eq:poissonLALAmap} (as Poisson LA-groupoids), and leads to 
an analog of Proposition~\ref{prop:PDAsymp} for LA-groupoids:

\begin{proposition}
A Poisson LA-groupoid $\Gamma$ as in \eqref{eq:poissonLAgrp} with $\pi$ {\em symplectic} is isomorphic to the Poisson LA-groupoid defined by the cotangent bundle $(T^*H,\omega_{can})$ of the bottom Poisson groupoid $H\toto M$. 
\end{proposition}

We now relate Poisson LA-groupoids to Poisson double algebroids.
Given $(\Gamma,\pi)$ a Poisson LA-groupoid as before, it automatically follows from the Lie theory of Poisson groupoids  and LA-groupoids 
that its differentiation,
\begin{equation}\label{eq:poissonLAGdiff}
\begin{matrix}
(\Gamma,\pi) & \rightrightarrows & A \\ 
\Downarrow &  & \Downarrow\\ 
H & \rightrightarrows & M
\end{matrix}
\qquad \mapsto \qquad
\begin{matrix}
(A_\Gamma,\mathrm{Lie}(\pi)) & \Rightarrow & A \\ 
\Downarrow &  & \Downarrow\\ 
A_H & \Rightarrow & M,
\end{matrix}
\end{equation}
is a double Lie algebroid whose top algebroid is a Poisson algebroid \cite{Mac-Xu,mac-double2},
but the compatibility between $A_\Gamma\then A_H$  and $\mathrm{Lie}(\pi)$ is not automatic and needs to be proven.
The same happens with integration: we know that if the (say, horizontal) top Lie algebroid of a Poisson double algebroid is integrable, then its source-simply-connected integration is an LA-groupoid \cite{BCdH} and a Poisson groupoid \cite{Mac-Xu2}, but it still remains to be verified that the integrated Poisson structure is compatible with the side Lie algebroid. In this context we have the following

\begin{theorem}\label{thm:diffintPLA}
\begin{enumerate}
    \item If $(\Gamma,\pi)$ is a Poisson LA-groupoid, then its differentiation $(A_\Gamma, \mathrm{Lie}(\pi))$ is a Poisson double algebroid.
    \item If $(\Omega,\pi_\Omega)$ is a Poisson double algebroid whose top algebroid is integrable, then its source-simply-connected integration  is a Poisson LA-groupoid.
\end{enumerate}
\end{theorem}

\begin{proof}
Let $\Gamma$ be an LA-groupoid and $\Omega=\mathrm{Lie}(\Gamma)$ be its infinitesimal double Lie algebroid. We know that $(\Gamma,\pi)$ is a Poisson LA-groupoid if and only if 
$\pi^\sharp:T\Gamma\to T^*\Gamma$ is an LA-groupoid morphism (with respect to \eqref{eq:LAGTT*}). Similarly, the Poisson structure  $\mathrm{Lie}(\pi)$ makes $\Omega$ into  a Poisson double algebroid if and only if $(\mathrm{Lie}(\pi))^\sharp=\mathrm{Lie}(\pi^\sharp) : T\Omega\to T^*\Omega$
is a double-Lie-algebroid morphism (with respect to \eqref{eq:DLATT*}).
Since $\Omega=\mathrm{Lie}(\Gamma)$ as double Lie algebroids, we also have natural identifications $\mathrm{Lie}(T^*\Gamma)=T^*\Omega$ and $\mathrm{Lie}(T\Gamma)=T\Omega$ as double Lie algebroids. 
The theorem now follows from Thm.~\ref{thm:Lie2LA}.
\end{proof}

\begin{remark}\label{rem:intPDA}
Building on Remark~\ref{rem:lie2more}, we observe that the arguments in the proof of Thm.~\ref{thm:Lie2LA} also imply the following general fact: If $\pi$ is a Poisson structure on a source-connected LA-groupoid $\Gamma$ \eqref{eq:LAgrp} such that  $(\Gamma\toto A,\pi)$ is a Poisson groupoid, then $(\Gamma,\pi)$ is a Poisson LA-groupoid provided  $(\Omega, \mathrm{Lie}(\pi))$ is a Poisson double algebroid. \hfill $\diamond$
\end{remark}

One can also reformulate Poisson LA-groupoids in the spirit of Def.~\ref{def:dlbalg}, as pairs of LA-groupoids $(\Gamma, \Gamma^*)$ which are in duality as VB-groupoids,
\begin{equation}\label{eq:LBgrp}
\begin{matrix}
\Gamma & \rightrightarrows & A \\ 
\Downarrow &  & \Downarrow\\ 
H & \rightrightarrows & M,
\end{matrix}
\qquad \;\; 
\begin{matrix}
\Gamma^* & \rightrightarrows & C^* \\ 
\Downarrow &  & \Downarrow\\ 
H & \rightrightarrows & M,
\end{matrix}
\end{equation}
and whose Lie-algebroid structures form a Lie bialgabroid $(\Gamma\then H, \Gamma^*\then H)$. Such pairs $(\Gamma,\Gamma^*)$ will be referred to as {\bf Lie-bialgebroid groupoids}.

We can now use Theorem~\ref{thm:diffintPLA} to relate Lie-bialgebroid groupoids and double Lie bialgebroids. Recall from \cite[Prop.~5.3.4]{BCdH} that if $\Gamma$ and $\Gamma^*$ are LA-groupoids in duality as VB-groupoids, then their corresponding double Lie algebroids $\Omega$ and $\Omega^*$ are in duality as VB-algebroids. By Theorem~\ref{thm:diffintPLA}, one can see that if $(\Gamma\then H, \Gamma^*\then H)$ is a Lie bialgebroid, then so is $(\Omega\then H, \Omega^*\then H)$, and the converse holds if $\Gamma$ is source-connected (see Remark~\ref{rem:intPDA}).

\begin{corollary}\label{cor:LBgdLB} If $\Gamma$ is a source-connected LA-groupoid, then $(\Gamma,\Gamma^*)$ is a Lie-bialgebroid groupoid if and only if $(\Omega,\Omega^*)$ is a double Lie bialgebroid.
\end{corollary}

\subsection{Poisson double groupoids}

A {\bf double Lie groupoid} (see \cite{brmk} and references therein) consists of a diagram of Lie groupoids
\begin{equation}\label{eq:DG}
\begin{matrix}
\Sigma & \rightrightarrows & G \\ 
\downdownarrows &  & \downdownarrows\\ 
H & \rightrightarrows & M
\end{matrix}
\end{equation}
such that (i) the  double source map $\Gamma\to G\times_M H$ is a surjective submersion, where $G\times_M H$ is the fiber product with respect to the source maps of $G$ and $H$, and (ii)
the horizontal groupoid structural maps are morphisms of the vertical Lie groupoids. 
In particular, for the multiplication map one must use the fact that the space of composable arrows 
$\Sigma\times_G \Sigma$ is a Lie groupoid over the composable arrows $H\times_M H$.
Condition (ii) has an equivalent version (ii') interchanging horizontal with vertical. 

A {\bf Poisson double groupoid} \cite{mk2} is a double Lie groupoid $\Sigma$ equipped with a Poisson structure $\pi$,
$$
\begin{matrix}
(\Sigma,\pi) & \rightrightarrows & G \\ \downdownarrows &  & \downdownarrows\\ 
H & \rightrightarrows & M,
\end{matrix}
$$
such that $(\Sigma,\pi)\rightrightarrows G$ and $(\Sigma,\pi)\rightrightarrows H$ are Poisson groupoids.
Alternatively, 
one can consider the natural double-Lie-groupoid structures on $T^*\Sigma$ and $T\Sigma$ (see \cite[Thm.~1.4]{mk2}) and require $\pi^\sharp: T^*\Sigma\to T\Sigma$ to be a morphism of double groupoids. 
As shown by Mackenzie \cite[Thm.~2.6]{mk2}, the Poisson structures induced on the sides of a double Poisson groupoid make them into Poisson groupoids, and the core inherits the structure of a Poisson groupoid as well.

Poisson double groupoids are common generalizations of Poisson groupoids  and symplectic double groupoids \cite{LuWe,We88}. Another class of examples is given by the so-called {\em Poisson 2-groups} \cite{CSX2}, that will be recalled in Section~\ref{sec:lie2bi}. New examples of Poisson double groupoids have been recently described in  \cite{daniel2}.

Mackenzie showed in \cite{mac-double1} that by differentiating the vertical groupoid structures in \eqref{eq:DG} we obtain an LA-groupoid:
\begin{equation}\label{eq:LieDGLAG}
\begin{matrix}
\Sigma & \rightrightarrows & G \\ 
\downdownarrows &  & \downdownarrows\\ 
H & \rightrightarrows & M
\end{matrix}\qquad \stackrel{\mathrm{Lie}}{\mapsto} \qquad 
\begin{matrix}
A_\Sigma & \rightrightarrows & A_G \\ 
\Downarrow &  & \Downarrow\\ 
H & \rightrightarrows & M.
\end{matrix}
\end{equation}
It follows that differentiating a Poisson double groupoid $(\Sigma,\pi)$ gives rise to an LA-groupoid $\Gamma$ equipped with a Poisson structure $\pi_\Gamma=\mathrm{Lie}(\pi)$ that is compatible with 
the Lie-algebroid structure, but one needs to prove that it is also compatible with the groupoid structure.


The integration of Poisson LA-groupoids is much subtler; in fact, even the integration of LA-groupoids is not fully understood yet (see e.g. \cite{luca,luca1}). For an LA-groupoid, its source-simply-connected integration may not carry a double Lie groupoid structure (see \cite[Ex.~2.4]{daniel}), but it is not known whether another integration with this property can always be found.  We will see in $\S$~\ref{subsec:Lie2alg} a class of LA-groupoids (``Lie 2-algebras'') for which this integration is always possible. 

We have the following result relating Poisson double groupoids and Poisson LA-groupoids.

\begin{theorem}\label{thm:diffintegpis}
\begin{enumerate}
    \item If $(\Sigma,\pi)$ is a Poisson double groupoid, then its differentiation $(\Gamma,\pi_\Gamma)$ is a Poisson LA-groupoid.
    \item  Let  $(\Gamma,\pi_\Gamma)$ be a Poisson LA-groupoid, and suppose that its source-simply-connected integration  $\Sigma$ is a double Lie groupoid such that the groupoid of composable pairs $\Sigma\times_G\Sigma \toto H\times_M H$ is source connected. Then the integrated Poisson structure $\pi$ makes $\Sigma$ into a Poisson double groupoid.
    \end{enumerate}
\end{theorem}

\begin{proof}
To prove (1), 
we must check that $(\Gamma,\pi_\Gamma)= (A_\Sigma,\mathrm{Lie}(\pi))$ is a Poisson groupoid over $A_G$. 
The fact that  $(\Sigma,\pi)\toto G$ is a Poisson groupoid means that the graph of the multiplication map $m_{\Sigma,G} : \Sigma\times_G \Sigma\to \Sigma$
is a coisotropic submanifold of $\Sigma\times \Sigma\times \overline{\Sigma}$, i.e., $\mathrm{graph}(m_{\Sigma,G}) \toto \mathrm{graph}(m_{H})$ is a coisotropic subgroupoid of $\Sigma\times \Sigma\times \overline{\Sigma} \toto H\times H \times H$.
The  composable pairs $A_\Sigma\times_{A_G}A_\Sigma$ define a Lie algebroid over $H\times_M H$ that is identified with $\mathrm{Lie}(\Sigma\times_G \Sigma)\then H\times_M H$ (see e.g. \cite[Prop.~A.3.1]{BCdH}), and the LA-groupoid multiplication $m_\Gamma$ agrees with
$\mathrm{Lie}(m_{\Sigma,G})$:
$$\begin{matrix}
\Sigma\times_G\Sigma & \stackrel{m_{\Sigma,G}}{\longrightarrow} & \Sigma \\ 
\downdownarrows &  & \downdownarrows\\ 
H\times_M H & \stackrel{m_H}{\longrightarrow} & H.
\end{matrix}
\qquad \mapsto \qquad
\begin{matrix}
A_\Sigma\times_{A_G}A_\Sigma & \stackrel{m_\Gamma}{\longrightarrow} & A_\Sigma \\ 
\Downarrow &  & \Downarrow\\ 
H\times_M H & \stackrel{m_H}{\longrightarrow} & H.
\end{matrix}
$$
By Proposition~\ref{prop:coisot}, $\mathrm{Lie}(\mathrm{graph}(m_{\Sigma,G}))=\mathrm{graph}(m_\Gamma)$ is a coisotropic subalgebroid of $\Gamma\times \Gamma\times \overline{\Gamma} \then H\times H\times H$. The fact that it is coisotropic says that $(\Gamma,\pi_\Gamma)\toto A_G$ is a Poisson groupoid,  hence $(\Gamma,\pi_\Gamma)$ is a Poisson LA-groupoid.

We prove (2) by reversing the arguments. 
Since $\Sigma\toto H$ is source simply connected, there exists a unique $\pi$ on $\Sigma$ such that $(\Sigma,\pi)\toto H$ is a Poisson groupoid and $\pi_\Gamma=\mathrm{Lie}(\pi)$.
Since $(\Gamma,\pi_\Gamma)$ is a Poisson LA-groupoid, then $\mathrm{graph}(m_\Gamma)=\mathrm{Lie}(\mathrm{graph}(m_{\Sigma,G}))$ is a coisotropic subalgebroid of $\Gamma\times \Gamma\times \overline{\Gamma} \then H\times H\times H$. 
Since we assume that $\Sigma\times_G\Sigma \toto H\times_M H$ is source connected, by Prop.~\ref{prop:coisot},   $\mathrm{graph}(m_{\Sigma,G}) \toto \mathrm{graph}(m_{H})$ is a coisotropic subgroupoid of $\Sigma\times \Sigma\times \overline{\Sigma} \toto H\times H \times H$, which implies that $(\Sigma,\pi)\toto G$ is a Poisson groupoid, and hence $(\Sigma,\pi)$ is a Poisson double groupoid.
\end{proof}

\begin{remark}
The source simply connectedness assumption in part (2) is required to integrate the Poisson structure. 
More generally, consider $\Sigma$ a double Lie groupoid and $\pi$ a Poisson structure turning $\Sigma \rightrightarrows H$ into a Poisson groupoid. When $\Sigma\times_G\Sigma \toto H\times_M H$ is source connected, $(\Sigma,\pi)$ is a Poisson double groupoid if and only if $(A_\Sigma,\mathrm{Lie}(\pi))$ is a Poisson LA-groupoid. \hfill $\diamond$
\end{remark}


The composition of \eqref{eq:LieDGLAG} and \eqref{eq:LAgrLie} gives a two-step differentiation procedure from double Lie groupoids to double Lie algebroids \cite{mac-double1,mac-double2}. 
By combining Theorems~\ref{thm:diffintegpis} and \ref{thm:diffintPLA} we obtain a two-step differentiation procedure from Poisson double groupoids to Poisson double algebroids, or equivalently, double Lie bialgebroids (as well as conditions for the corresponding integrations):
$$
\xymatrix@R=10pt{
\mbox{Poisson double groupoids} \ar[d]^{\mathrm{Lie}}\\ 
\mbox{Poisson LA-groupoids} \ar[d]^{\mathrm{Lie}} \\ 
\mbox{Poisson double algebroids $\Leftrightarrow$ double Lie bialgebroids}.
}
$$
If one instead uses horizontal differentiation in the first step, the resulting double Lie algebroid is unchanged, up to natural isomorphism, and the resulting Poisson double algebroids will be naturally isomorphic too.

Starting with a {\em symplectic} double groupoid, the two-step differentiation leads to a {\em symplectic} double algebroid, which is necessarily the cotangent bundle of a Lie bialgebroid by Proposition~\ref{prop:PDAsymp}. 
This recovers  \cite[Thm.2.11]{mk2} and implies the following result, cf.  \cite[Thm.~2.9]{mk2}.

\begin{corollary}
 For a symplectic double groupoid, the Lie bialgebroids corresponding to its side Poisson groupoids are in duality.
 \end{corollary}

Although our results give information about the integration of Lie bialgebroids to symplectic double groupoids, Theorem~ \ref{thm:diffintegpis} involves hypothesis which restrict its applicability (cf. \cite[Sec.~2.4]{luca} for another approach). It would be interesting to prove a general version in the context of ``local'' integration (as e.g. in \cite{CMS}).



\section{Application: another look at Lie 2-bialgebras}\label{sec:lie2bi}

This section concerns double structures with a {\em trivial side}, which provide an approach to Lie 2-bialgebras \cite{BSZ,CSX} via Poisson double structures. 
We start by revisiting  Lie 2-algebras (see e.g. \cite{BaezCrans}) in light of the objects discussed in Section~\ref{sec:DS}.
When referring to Lie 2-(bi)algebras we will always mean their ``strict'' versions (though our methods can be adapted to handle weaker notions).

\subsection{Lie 2-algebras revisited}\label{subsec:Lie2alg}
A {\bf 2-vector space} is a  VB-groupoid
$$
\begin{matrix}
\Gamma & \rightrightarrows & \mathfrak{g}_0 \\ 
\downarrow &  & \downarrow\\ 
* & \rightrightarrows & *.
\end{matrix}
$$
As recalled in Example~\ref{ex:VBgr2term},  $\Gamma$ is encoded on two vector spaces $\mathfrak{g}_0$ and $\mathfrak{g}_1$ (its side and core) and a linear map $\partial: \mathfrak{g}_1\to \mathfrak{g}_0$, since $\Gamma$ is the action groupoid for the action of $\mathfrak{g}_1$ (viewed as an abelian group) on $\mathfrak{g}_0$ by $c\cdot v = v+\partial(c)$. 
Along these lines \cite{BaezCrans}, a {\bf Lie 2-algebra} is an LA-groupoid of type
\begin{equation}\label{eq:Lie2algebra}
\begin{matrix}
\Gamma & \rightrightarrows & \mathfrak{g}_0 \\ 
\Downarrow &  & \Downarrow\\ 
* & \rightrightarrows & *.
\end{matrix}
\end{equation}

By means of differentiation and integration, Lie 2-algebras are in correspondence with special types of double Lie groupoids and double Lie algebroids:
\begin{equation}\label{eq:Lie2algequiv}
\begin{matrix}
{G} & \toto & G_0 \\ \downdownarrows & & \downdownarrows\\ * & \toto & *
\end{matrix} \quad \;\; \leftrightharpoons \quad \;\; 
\begin{matrix}
\Gamma & \rightrightarrows & \mathfrak{g}_0 \\ 
\Downarrow &  & \Downarrow\\ 
* & \rightrightarrows & *
\end{matrix} \quad \;\;  \leftrightharpoons \quad \;\;
\begin{matrix}
\mathfrak{g} & \Longrightarrow & \mathfrak{g}_0 \\ 
\Downarrow &  & \Downarrow\\ 
* & \Longrightarrow & *.
\end{matrix}
\end{equation}

The correspondence on the right is a special instance of the equivalence between double Lie algebroids and (source simply connected) LA-groupoids  \cite[Thm.~5.3.5]{BCdH}. 
The next result explicitly describes the information codified on the infinitesimal side.

\begin{proposition}\label{prop:dlacrossed}
A double-Lie-algebroid structure of type  \begin{equation}\label{eq:DLAcm}
    \begin{matrix}
\mathfrak{g} & \Longrightarrow & \mathfrak{g}_0 \\ \Downarrow & & \Downarrow\\ * & \Longrightarrow & *,
\end{matrix}
\end{equation}
is equivalent to the following data: a Lie algebra $\mathfrak{g}_0$, a vector space $\mathfrak{g}_1$, a linear map $\partial:\mathfrak{g}_1 {\to} \mathfrak{g}_0$ and a representation $\mathfrak{g}_0\to \mathrm{End}(\mathfrak{g}_1)$, such that
 \begin{itemize}
     \item[(a)] $\partial(v\cdot c) = [v,\partial(c)]_0$, for $v\in \mathfrak{g}_0$, $c\in \mathfrak{g}_1$,
     \item[(b)] $(\partial c_1)  \cdot c_2 = - (\partial c_2) \cdot c_1$, for $c_1,c_2 \in \mathfrak{g}_1$,
 \end{itemize}
where ``$\cdot$'' denotes the $\mathfrak{g}_0$-ation on $\mathfrak{g}_1$.
\end{proposition}

\begin{proof}
As a double vector bundle, $\mathfrak{g}=\mathfrak{g}_0\oplus \mathfrak{g}_1$, the direct sum of its side and core vector spaces.
We know from Example~\ref{ex:LArepresentations} (a) that the vertical VB-algebroid structure on \eqref{eq:DLAcm}
is equivalent to a Lie bracket on $\mathfrak{g}_0$ and Lie-algebra representation  $\mathfrak{g}_0\to \mathrm{End}(\mathfrak{g_1})$, 
while 
Example ~\ref{ex:2term} says that  the horizontal VB-algebroid structure is equivalent to a linear map $\partial: \mathfrak{g}_1 {\to} \mathfrak{g}_0$.
The compatibility condition making \eqref{eq:DLAcm} into a double Lie algebroid can be checked to agree with conditions (a) and (b)  (which also follows from \cite[Thm.~3.5]{GJMM}).
\end{proof}

The data arising in the previous proposition can be interpreted as follows. A {\bf Lie-algebra crossed module} (see e.g. \cite{BaezCrans}), denoted by $[\mathfrak{g}_1\stackrel{\partial}{\to}\mathfrak{g}_0]$, consists of Lie algebras $\mathfrak{g}_0$ and $\mathfrak{g}_1$, a Lie algebra morphism $\partial: \mathfrak{g}_1\to \mathfrak{g}_0$ and an action of $\mathfrak{g}_0$ on $\mathfrak{g}_1$ by derivations, written as $(v,c)\mapsto v\cdot c$, for $v\in \mathfrak{g}_0$ and $c\in \mathfrak{g}_1$, such that  
\begin{itemize}
\item $\partial(v\cdot c)  = [v,\partial(c)]_0, \;\; \mbox{ for } v\in \mathfrak{g}_0, \, c\in \mathfrak{g}_1$,
\item   $ \partial(c_1)\cdot c_2  = [c_1,c_2]_1, \;\; \mbox{ for }\, c_1,c_2\in \mathfrak{g}_1.$ 
\end{itemize}
To see that the (seemingly more general) data in Prop.~\ref{prop:dlacrossed} is equivalent to that of a crossed module (cf. \cite[Prop.~ 3.3]{CSX}), note that by (b) 
we can equip $\mathfrak{g}_1$ with the skew-symmetric bracket $[c_1,c_2]_1:=(\partial c_1)  \cdot c_2$, and (a) implies that $\partial([c_1,c_2]_1)=[\partial c_1, \partial c_2]_0$. 
The Jacobi identity follows from the fact that the $\mathfrak{g}_0$-action on $\mathfrak{g}_1$ is by derivations of $[\cdot,\cdot]_1$, which we now check. For $v\in \mathfrak{g}_0$ and $c_1,c_2\in \mathfrak{g}_1$, we have 
$$
[v\cdot c_1,c_2]=(\partial(v\cdot c_1)) \cdot c_2 = [v,\partial c_1]_0 \cdot c_2 = v\cdot (\partial c_1 \cdot c_2)- \partial c_1 \cdot (v\cdot c_2),
$$
where we have used (a) and that $\mathfrak{g}_0\to \mathrm{End}(\mathfrak{g}_1)$ is bracket preserving. Finally
$$
v\cdot [c_1,c_2]_1= v\cdot (\partial c_1\cdot c_2)= [v\cdot c_1,c_2]_1 + \partial c_1 \cdot (v\cdot c_2) = [v\cdot c_1,c_2]_1 + [c_1, v\cdot c_2]_1.
$$

In conclusion, the right correspondence in \eqref{eq:Lie2algequiv} and the previous proposition just give the well-known equivalence of Lie 2-algebras and Lie-algebra crossed modules.

Note that both $\Gamma$ and $\mathfrak{g}$ are naturally identified, as vector spaces,  with $\mathfrak{g}_1\oplus \mathfrak{g}_0$, where $\mathfrak{g}_1$ is the core of each diagram.
But they differ as (vertical) Lie algebras: 
$\Gamma$ is the semi-direct product of $\mathfrak{g}_0$ and $\mathfrak{g}_1$ as Lie algebras, whereas 
$\mathfrak{g}$ is the semi-direct product of $\mathfrak{g}_0$ with the $\mathfrak{g}_0$-module $\mathfrak{g}_1$, i.e., forgetting the Lie-algebra structure on $\mathfrak{g}_1$.

The algebraic approach to double Lie algebroids in terms of Weil algebras \cite{pike} gives another viewpoint on Lie 2-algebras.
Any 2-vector space $\Gamma$ with side $\mathfrak{g}_0$ and core $\mathfrak{g}_1$ has an underlying (split) double vector bundle $\mathfrak{g}=\mathfrak{g}_0\oplus \mathfrak{g}_1$.
As recalled in $\S$~\ref{subsec:weil}, its corresponding Weil algebra is simply 
\begin{equation}\label{eq:weil}
    \mathcal{W}(\mathfrak{g})= \wedge \mathfrak{g}_0^* \otimes \vee \mathfrak{g}_1^*,
\end{equation}    
seen as a bigraded algebra where elements in $\mathfrak{g}_0^*$ have bidegree $(1,0)$ and those in $\mathfrak{g}_1^*$ have bidegree $(1,1)$. The map $\partial: \mathfrak{g}_1\to \mathfrak{g}_0$ determining the 2-vector space structure on $\Gamma$ (or, equivalently, the horizontal VB-algebroid structure on $\mathfrak{g}$)  is encoded in a degree $(0,1)$ differential $\delta_v$ on $\mathcal{W}(\mathfrak{g})$, the unique extension of $\partial^*: \mathfrak{g}_0^*\to \mathfrak{g}_1^*$. 

The additional Lie-algebra structure making $\Gamma$ into a Lie 2-algebra is given by a second differential $\delta_h$ on $\mathcal{W}(\mathfrak{g})$, of bidegree $(1,0)$, commuting with $\delta_v$. Note that $\delta_h$
is determined by the maps 
$$
\delta_h: \mathfrak{g}_0^*\to \wedge^2\mathfrak{g}_0^*,\qquad
\delta_h: \mathfrak{g}_1^*\to \mathfrak{g}_0^*\otimes \mathfrak{g}_1^*. 
$$
The dual maps give the Lie bracket on $\mathfrak{g}_0$ and the representation of $\mathfrak{g}_0$ on $\mathfrak{g}_1$ of Proposition~\ref{prop:dlacrossed},
$$
[\cdot,\cdot]:  \wedge^2\mathfrak{g}_0 \to \mathfrak{g}_0^,\qquad
[\cdot,\cdot]: \mathfrak{g}_0\otimes \mathfrak{g}_1 \to \mathfrak{g}_1,
$$
and extend to a Gerstenhaber bracket $[\cdot,\cdot]$ on $\mathcal{W}(\mathfrak{g}^*)= \wedge \mathfrak{g}_1 \otimes \vee \mathfrak{g}_0$ compatible with the differential extending $\partial:  \mathfrak{g}_1\to \mathfrak{g}_0 $

\begin{remark}[Weak Lie 2-algebras]
The differentials $\delta_h$ and $\delta_v$ on $\mathcal{W}(\mathfrak{g})$ can be seen as components of a differential of total degree 1. 
In general, 
a {\bf weak Lie 2-algebra} (or a 2-term $L_\infty$-algebra)  \cite{BaezCrans,BSZ}
is a 2-term graded vector space $\mathfrak{g}=\mathfrak{g}_1\oplus \mathfrak{g}_0$ with a degree 1 differential $\delta$ on the (graded) symmetric algebra
${\rm Sym}(\mathfrak{g}^*[-1])\cong \mathcal{W}(\mathfrak{g})$, where we regard $\mathfrak{g}$ as a double vector bundle with core $\mathfrak{g}_1$ and (right) side $\mathfrak{g}_0$. Besides the horizontal and vertical component, $\delta$ has an additional component of bidegree $(2,-1)$, determined by the {\em Jacobiator} $\delta_J:\mathfrak{g}_1^*\to \wedge^3\mathfrak{g}_0^*$.
In conclusion, a Lie 2-algebra is the same as a weak Lie 2-algebra whose Jacobiator vanishes (\cite[Sec.~5]{BaezCrans}).
\hfill $\diamond$
\end{remark}

We finally comment on the left correspondence in \eqref{eq:Lie2algequiv}. The double Lie groupoids on the left-hand side are called {\bf Lie 2-groups}.  
These objects are equivalent to Lie-group crossed modules \cite{brown} (see also \cite{noohi}), which are the natural global counterparts of the Lie-algebra  crossed modules.
From this perspective, by differentiating and integrating crossed-module data,
one obtains a bijective correspondence between Lie 2-algebras and {\em simply connected} Lie 2-groups (i.e., Lie 2-groups $G$ whose side $G_0$ and core $G_1 = \mathrm{ker}(s: G\to G_0)$ are simply connected; note that, as a manifold, $G=G_0\times G_1$, so this implies that $G$ is also simply connected).

\subsection{Lie 2-bialgebras and Poisson double structures} \label{subsec:lie2bi}
Our goal now is to relate Poisson double structures to Lie 2-bialgebras \cite{BSZ,CSX}, which we introduce as special cases of the ``Lie- bialgebroid groupoids'' \eqref{eq:LBgrp} of $\S$~\ref{subsec:PLA}. 

\begin{definition}\label{def:Lie2bi}
A {\bf Lie 2-bialgebra} is a pair of Lie 2-algebras $(\Gamma,\Gamma^*)$ on dual 2-vector spaces (duality as VB-groupoids) such that, as Lie algebras, they form a Lie bialgebra.
\end{definition}

By Corollary~\ref{cor:LBgdLB}, Lie 2-bialgebras can be equivalently defined by their infinitesimal counterparts (in the sense of Definition~\ref{def:dlbalg}).

Using the equivalence between Lie 2-algebras and Lie-algebra crossed modules, we see that the previous definition is equivalent to that of a {\em Lie-bialgebra crossed module}  \cite[Def.~3.4]{CSX}, and it follows from  \cite[Thm.~3.8]{CSX}
that our definition of Lie 2-bialgebra agrees with the one in \cite{CSX2,CSX}.

By dualizing the Lie-algebra structure on $\Gamma^*$,  Lie 2-bialgebras are seen to be equivalent to 
{\bf Poisson Lie 2-algebras}, i.e., Poisson LA-groupoids of the form
$$
 \begin{matrix}
(\Gamma, \pi) & \toto & \mathfrak{g}_0 \\ \Downarrow & & \Downarrow\\ * & \toto & *.
\end{matrix}
$$

Our results in Section~\ref{sec:diffint} on Lie theory of Poisson double structures lead to a refinement of the equivalences in 
\eqref{eq:Lie2algequiv}:
\begin{equation}\label{eq:PDSequiv}
\begin{matrix}
(G,\pi_G) & \toto & G_0 \\ \downdownarrows & & \downdownarrows\\ * & \toto & *
\end{matrix} \quad \;\; \leftrightharpoons \quad \;\; 
\begin{matrix}
(\Gamma,\pi_\Gamma) & \rightrightarrows & \mathfrak{g}_0 \\ 
\Downarrow &  & \Downarrow\\ 
* & \rightrightarrows & *
\end{matrix} \quad \;\;  \leftrightharpoons \quad \;\;
\begin{matrix}
(\mathfrak{g},\pi) & \Longrightarrow & \mathfrak{g}_0 \\ 
\Downarrow &  & \Downarrow\\ 
* & \Longrightarrow & *.
\end{matrix}
\end{equation}


A {\bf Poisson 2-group} is a Poisson double groupoid as in the left. Their correspondence with Poisson 2-algebras is a special case of Theorem~\ref{thm:diffintegpis} and immediately leads to the following result \cite[Thm.~3.2, Cor.~3.3]{CSX2}.

\begin{proposition}
A Poisson 2-group differentiates to a Lie 2-bialgebra, 
any Lie 2-bialgebra can be integrated to a Lie 2-group, and this establishes a one-to-one correspondence between simply connected Poisson 2-groups and Lie 2-bialgebras.
\end{proposition}

\begin{proof} 
We regard a Lie 2-bialgebra as a Poisson Lie 2-algebra. We know that any Lie 2-algebra can be integrated to a simply connected 2-Lie group. Recall (see e.g. \cite{noohi}) that, on a Lie 2-group $G$, the groupoid structure  is that of an action Lie groupoid $G_1\ltimes G_0 \toto G_0$, where $G_1$ is the core.
Assuming that $G_0$ and $G_1$ are simply connected, $G\toto G_0$ is source simply connected and
the space of composable pairs  $G\times_{G_0}G = G_1\times G_0 \times G_1$ is simply connected.
The proposition then follows from Theorem~\ref{thm:diffintegpis}.
\end{proof}

Regarding the right-hand side of \eqref{eq:PDSequiv}, since Lie 2-algebras are source simply connected, we have an equivalence  (by differentiation and integration) between Poisson Lie 2-algebras and their corresponding Poisson double algebroids by Theorem~\ref{thm:diffintPLA}.
The next result describes the infinitesimal object, leading to an alternative characterization of Lie 2-bialgebras.

\begin{proposition}\label{prop:matched}
A Poisson double algebroid of type
\begin{equation}\label{eq:PDLA}
    \begin{matrix}
(\mathfrak{g},\pi) & \Longrightarrow & \mathfrak{g}_0 \\ \Downarrow & & \Downarrow\\ * & \Longrightarrow & *
\end{matrix}
\end{equation}
is equivalent to Lie algebra
 crossed modules $[\mathfrak{g}_1\stackrel{\partial}{\to} \mathfrak{g}_0]$ and $[\mathfrak{g}_0^*\stackrel{\partial^*}{\to} \mathfrak{g}^*_1]$ such that $(\mathfrak{g}_0,\mathfrak{g}_1^*)$ forms a matched pair of Lie algebras (with respect to the dual actions).
\end{proposition}

\begin{proof}
As we saw in $\S$ \ref{subsec:Lie2alg} the double Lie algebroid in \eqref{eq:PDLA} is equivalent to a crossed module $[\mathfrak{g}_1\stackrel{\partial}{\to} \mathfrak{g}_0]$, where $\mathfrak{g}_1$ is the core. The compatibility condition with $\pi$ is equivalent to the  horizontal and vertical VB-algebroids in \eqref{eq:PDLA} being Poisson VB-algebroids: 
$$
\begin{matrix}
(\mathfrak{g},\pi) & \Longrightarrow & \mathfrak{g}_0 \\ \downarrow & & \downarrow\\ * & \Longrightarrow & *,
\end{matrix} \qquad\qquad
\begin{matrix}
(\mathfrak{g},\pi) & \longrightarrow & \mathfrak{g}_0 \\ \Downarrow & & \Downarrow\\ * & \longrightarrow & *,
\end{matrix} 
$$
or, equivalently, that $\pi$ induces double-Lie-algebroid structures on the  duals:
$$
\begin{matrix}
\mathfrak{g}^* & \Longrightarrow & \mathfrak{g}_1^* \\ \Downarrow & & \Downarrow\\ * & \Longrightarrow & *,
\end{matrix}\qquad\qquad 
\begin{matrix}
\mathfrak{g}^\bullet & \Longrightarrow & \mathfrak{g}_0 \\ \Downarrow & & \Downarrow\\ \mathfrak{g}_1^* & \Longrightarrow & *.
\end{matrix}
$$
The left diagram again defines a Lie-algebra crossed module   $[\mathfrak{g}_0^*\stackrel{\partial^*}{\to}\mathfrak{g}_1^*]$, while by Example~\ref{ex:matchedpair} the right diagram says that the representations of $\mathfrak{g}_0$ on $\mathfrak{g}_1^*$ and of  $\mathfrak{g}_1^*$ on $\mathfrak{g}_0$ form a matched pair.
\end{proof}

Since Lie 2-bialgebras are equivalent to Poisson double algebroids \eqref{eq:PDLA},
the previous proposition recovers the characterization of Lie 2-bialgebras in \cite[Thm.~3.12]{CSX2}.


The algebraic characterization of Poisson double algebroids via Weil algebras in
$\S$ \ref{subsec:weil} provides us with an  alternative definition of Lie 2-bialgebra, parallel to the descriptions of usual Lie bialgebras in $\S$ \ref{subsec:poisalg}.
Following the notation in $\S$ \ref{subsec:Lie2alg}, for a 2-vector space $\Gamma$ we have a corresponding Weil algebra 
$\mathcal{W}(\mathfrak{g})$ \eqref{eq:weil}, where $\mathfrak{g}= \mathfrak{g}_0\oplus \mathfrak{g}_1$ is its underlying double vector bundle.

\begin{proposition}
A Lie 2-bialgebra $(\Gamma,\Gamma^*)$ consists of Lie 2-algebra structures on dual 2-vector spaces such that the differential $\delta_h+\delta_v$ on $\mathcal{W}(\mathfrak{g})$ induced by $\Gamma$ is a derivation of the Gerstenhaber bracket $[\cdot,\cdot]$ induced by $\Gamma^*$.
\end{proposition}

One can also view a Lie 2-bialgebra as given by two vector spaces $\mathfrak{g}_0$ and $\mathfrak{g}_1$ and the data $(\mathcal{W}(\mathfrak{g}), \delta_h, \delta_v, [\cdot,\cdot])$ as described in Proposition~\ref{prop:pda-dlba}.

This algebraic viewpoint can be used to see the equivalence with the Lie 2-bialgebras in 
\cite[Section~3]{BSZ} (the conditions displayed in equations (12) and (17) therein correspond to the characterization via matched pairs).
An attractive feature of the categorification of Lie bialgebras via double structures is that all the gradings are naturally built in from the duality theory of VB-groupoids, VB-algebroids and double vector bundles (cf. \cite[$\S$ 2.3]{BSZ} and \cite[$\S$ 2.1]{CSX}).


\end{document}